\documentclass[10pt]{article}
\usepackage[skaknew]{skak}
\usepackage{charter}\usepackage[T1]{fontenc}\usepackage{textcomp}
\usepackage[scaled=.92]{helvet}
\usepackage{eulervm,euscript}
\DeclareFontFamily{T1}{pzc}{}
\DeclareFontShape{T1}{pzc}{m}{it}{<-> s * [1.200] pzcmi}{}
\DeclareMathAlphabet{\nicemathcal}{OT1}{pzc}{m}{it}
\DeclareMathAlphabet{\oldmathcal}{OMS}{bch}{m}{n}
\usepackage{mathrsfs}
\usepackage{amsmath,amssymb,geometry,theorem}
\usepackage{pstricks,pst-node,pst-text,pst-tree,pstricks-add}
\psset{radius=2pt,xunit=.5cm,yunit=1,arrowscale=1.7,arrowlength=.7,labelsep=3pt,rowsep=2\baselineskip}

\makeatletter
\def\section{\@startsection{section}{1}{0pt}{-3.25ex plus -1ex minus 
-.2ex}{1.5ex plus .2ex minus .3ex}{\normalfont\large\bf}}
\renewcommand\subsection{\@startsection{subsection}{2}{\z@}%
                                     {-3.25ex\@plus -1ex \@minus -.2ex}%
                                     {1.5ex \@plus .2ex}%
                                     {\normalfont\normalsize\bfseries}}
\makeatother
\renewenvironment{abstract}
{\vspace*{-1.8ex}\begin{quotation}\small
}{\end{quotation}}

\newcommand{\defn}[1]{{\textit{\textbf{#1}}}}
\newcommand{\myitem}[1]{\item[\textnormal{(#1)}]}

\theoremheaderfont{\scshape}
\theorembodyfont{\normalfont\slshape}
\theoremstyle{plain}

\newtheorem{theorem}{Theorem}
\theorembodyfont{\normalfont}

\newenvironment{proof}{\begin{trivlist}\item{}\normalfont\textit{Proof.}}{\hfill$\square$\end{trivlist}}

\newcommand{\ie}{\emph{i.e.}}
\newcommand{\eg}{\emph{e.g.}}
\newcommand{\cf}{\emph{cf.}}
\newcommand{\etcet}{\emph{etc.}}

\newdimen\arrayruleHwidth
\makeatletter
\def\Hline{\noalign{\ifnum0=`}\fi\hrule \@height \arrayruleHwidth
  \futurelet \@tempa\@xhline}
\makeatother

\newdimen\proofrulebreadth \proofrulebreadth=.05em
\newdimen\proofdotseparation \proofdotseparation=1.25ex
\newdimen\proofrulebaseline \proofrulebaseline=2ex
\newcount\proofdotnumber \proofdotnumber=3
\let\then\relax
\def\hfi{\hskip0pt plus.0001fil}
\mathchardef\squigto="3A3B
\newif\ifinsideprooftree\insideprooftreefalse
\newif\ifonleftofproofrule\onleftofproofrulefalse
\newif\ifproofdots\proofdotsfalse
\newif\ifdoubleproof\doubleprooffalse
\let\wereinproofbit\relax
\newdimen\shortenproofleft
\newdimen\shortenproofright
\newdimen\proofbelowshift
\newbox\proofabove
\newbox\proofbelow
\newbox\proofrulename
\def\shiftproofbelow{\let\next\relax\afterassignment\setshiftproofbelow\dimen0 }
\def\shiftproofbelowneg{\def\next{\multiply\dimen0 by-1 }%
\afterassignment\setshiftproofbelow\dimen0 }
\def\setshiftproofbelow{\next\proofbelowshift=\dimen0 }
\def\setproofrulebreadth{\proofrulebreadth}

\def\prooftree{
\ifnum  \lastpenalty=1
\then   \unpenalty
\else   \onleftofproofrulefalse
\fi
\ifonleftofproofrule
\else   \ifinsideprooftree
        \then   \hskip.5em plus1fil
        \fi
\fi
\bgroup
\setbox\proofbelow=\hbox{}\setbox\proofrulename=\hbox{}%
\let\justifies\proofover\let\leadsto\proofoverdots\let\Justifies\proofoverdbl
\let\using\proofusing\let\[\prooftree
\ifinsideprooftree\let\]\endprooftree\fi
\proofdotsfalse\doubleprooffalse
\let\thickness\setproofrulebreadth
\let\shiftright\shiftproofbelow \let\shift\shiftproofbelow
\let\shiftleft\shiftproofbelowneg
\let\ifwasinsideprooftree\ifinsideprooftree
\insideprooftreetrue
\setbox\proofabove=\hbox\bgroup$\displaystyle 
\let\wereinproofbit\prooftree
\shortenproofleft=0pt \shortenproofright=0pt \proofbelowshift=0pt
\onleftofproofruletrue\penalty1
}

\def\eproofbit{
\ifx    \wereinproofbit\prooftree
\then   \ifcase \lastpenalty
        \then   \shortenproofright=0pt  
        \or     \unpenalty\hfil         
        \or     \unpenalty\unskip       
        \else   \shortenproofright=0pt  
        \fi
\fi
\global\dimen0=\shortenproofleft
\global\dimen1=\shortenproofright
\global\dimen2=\proofrulebreadth
\global\dimen3=\proofbelowshift
\global\dimen4=\proofdotseparation
\global\count255=\proofdotnumber
$\egroup  
\shortenproofleft=\dimen0
\shortenproofright=\dimen1
\proofrulebreadth=\dimen2
\proofbelowshift=\dimen3
\proofdotseparation=\dimen4
\proofdotnumber=\count255
}

\def\proofover{
\eproofbit 
\setbox\proofbelow=\hbox\bgroup 
\let\wereinproofbit\proofover
$\displaystyle
}%
\def\proofoverdbl{
\eproofbit 
\doubleprooftrue
\setbox\proofbelow=\hbox\bgroup 
\let\wereinproofbit\proofoverdbl
$\displaystyle
}%
\def\proofoverdots{
\eproofbit 
\proofdotstrue
\setbox\proofbelow=\hbox\bgroup 
\let\wereinproofbit\proofoverdots
$\displaystyle
}%
\def\proofusing{
\eproofbit 
\setbox\proofrulename=\hbox\bgroup 
\let\wereinproofbit\proofusing
\kern0.3em$
}

\def\endprooftree{
\eproofbit 
  \dimen5 =0pt
\dimen0=\wd\proofabove \advance\dimen0-\shortenproofleft
\advance\dimen0-\shortenproofright
\dimen1=.5\dimen0 \advance\dimen1-.5\wd\proofbelow
\dimen4=\dimen1
\advance\dimen1\proofbelowshift \advance\dimen4-\proofbelowshift
\ifdim  \dimen1<0pt
\then   \advance\shortenproofleft\dimen1
        \advance\dimen0-\dimen1
        \dimen1=0pt
        \ifdim  \shortenproofleft<0pt
        \then   \setbox\proofabove=\hbox{%
                        \kern-\shortenproofleft\unhbox\proofabove}%
                \shortenproofleft=0pt
        \fi
\fi
\ifdim  \dimen4<0pt
\then   \advance\shortenproofright\dimen4
        \advance\dimen0-\dimen4
        \dimen4=0pt
\fi
\ifdim  \shortenproofright<\wd\proofrulename
\then   \shortenproofright=\wd\proofrulename
\fi
\dimen2=\shortenproofleft \advance\dimen2 by\dimen1
\dimen3=\shortenproofright\advance\dimen3 by\dimen4
\ifproofdots
\then
        \dimen6=\shortenproofleft \advance\dimen6 .5\dimen0
        \setbox1=\vbox to\proofdotseparation{\vss\hbox{$\cdot$}\vss}%
        \setbox0=\hbox{%
                \advance\dimen6-.5\wd1
                \kern\dimen6
                $\vcenter to\proofdotnumber\proofdotseparation
                        {\leaders\box1\vfill}$%
                \unhbox\proofrulename}%
\else   \dimen6=\fontdimen22\the\textfont2 
        \dimen7=\dimen6
        \advance\dimen6by.5\proofrulebreadth
        \advance\dimen7by-.5\proofrulebreadth
        \setbox0=\hbox{%
                \kern\shortenproofleft
                \ifdoubleproof
                \then   \hbox to\dimen0{%
                        $\mathsurround0pt\mathord=\mkern-6mu%
                        \cleaders\hbox{$\mkern-2mu=\mkern-2mu$}\hfill
                        \mkern-6mu\mathord=$}%
                \else   \vrule height\dimen6 depth-\dimen7 width\dimen0
                \fi
                \unhbox\proofrulename}%
        \ht0=\dimen6 \dp0=-\dimen7
\fi
\let\doll\relax
\ifwasinsideprooftree
\then   \let\VBOX\vbox
\else   \ifmmode\else$\let\doll=$\fi
        \let\VBOX\vcenter
\fi
\VBOX   {\baselineskip\proofrulebaseline \lineskip.2ex
        \expandafter\lineskiplimit\ifproofdots0ex\else-0.6ex\fi
        \hbox   spread\dimen5   {\hfi\unhbox\proofabove\hfi}%
        \hbox{\box0}%
        \hbox   {\kern\dimen2 \box\proofbelow}}\doll%
\global\dimen2=\dimen2
\global\dimen3=\dimen3
\egroup 
\ifonleftofproofrule
\then   \shortenproofleft=\dimen2
\fi
\shortenproofright=\dimen3
\onleftofproofrulefalse
\ifinsideprooftree
\then   \hskip.5em plus 1fil \penalty2
\fi
}

\newcommand{\trans}{\begin{psmatrix}[colsep=2.4ex]\rnode{l}{\rule{0pt}{1.2em}}&\rnode{r}{\rule{0pt}{1.2em}}\ncline[arrows=->,nodesep=2pt]{l}{r}\end{psmatrix}}
\newcommand{\transbetweena}[4]{\ncline[arrows=->]{#1}{#2}\aput(#4){\mbox{\small$#3$}}}
\newcommand{\transbetweenb}[4]{\ncline[arrows=->]{#1}{#2}\bput(#4){\mbox{\small$#3$}}}

\newcommand{\transto}[1]{\begin{psmatrix}[labelsep=1.5pt,colsep=4.5ex]\rnode{l}{\rule{0pt}{1.1em}\,}&\rnode{r}{\,\rule{0pt}{1.1em}}\ncline[arrows=->,nodesep=2pt]{l}{r}\aput(.4){\mbox{\small$#1$}}\end{psmatrix}}

\newcommand{\opp}{\mathsf{O}}
\newcommand{\pla}{\mathsf{P}}

\newcommand{\player}{\pla}

\newcommand{\jptr}[4]{\nccurve[angleA=#3,angleB=#4]{#1}{#2}}

\newcommand{\pto}{\rightharpoonup}

\newrgbcolor{ocolor}{.85 .35 0}
\newrgbcolor{pcolor}{.35 0 .5}

\newcommand{\backtrack}[1]{\widehat{#1}}
\newcommand{\N}{\mathbb{N}}

\newcommand{\Nat}{\N}

\newcommand{\emptyseq}{\varepsilon}

\newcommand{\pbacktrack}[1]{\backtrack{#1}^{\pla}}
\newcommand{\bbgame}[1]{\backtrack{#1}^{\mathsf{B}}}

\newcommand{\chess}{\textsf{Chess}}
\newcommand{\oxX}{\textsf{\textbf X}}
\newcommand{\oxO}{\textsf{\textbf O}}
\newcommand{\ox}{\textsf{\oxO's \& \oxX's}}

\newcommand{\simul}[1]{#1\mkern-2mu\to\mkern-2mu#1}

\newcommand{\oxgrid}[9]{
\thicklines\setlength{\unitlength}{2.8pt}%
\begin{picture}(15,15)(-7.5,-7.5)%
\put(-2.5,7.5){\line(0,-1){15}}\put(2.5,7.5){\line(0,-1){15}}\put(-7.5,2.5){\line(1,0){15}}\put(-7.5,-2.5){\line(1,0){15}}%
\put(-5,5){\makebox(0,0){\textsf{\textbf  #1}}}%
\put(0,5){\makebox(0,0){\textsf{\textbf  #2}}}%
\put(5,5){\makebox(0,0){\textsf{\textbf  #3}}}%
\put(-5,0){\makebox(0,0){\textsf{\textbf  #4}}}%
\put(0,0){\makebox(0,0){\textsf{\textbf  #5}}}%
\put(5,0){\makebox(0,0){\textsf{\textbf  #6}}}%
\put(-5,-5){\makebox(0,0){\textsf{\textbf  #7}}}%
\put(0,-5){\makebox(0,0){\textsf{\textbf  #8}}}%
\put(5,-5){\makebox(0,0){\textsf{\textbf  #9}}}%
\end{picture}}

\newcommand{\oxempty}{\oxgrid{}{}{}{}{}{}{}{}{}}
\newcommand{\oxcentrex}{\oxgrid{}{}{}{}{X}{}{}{}{}}

\newcommand{\chessposn}[1]{\chessposnscale{0.26}{#1}}

\newcommand{\chessposnscale}[2]{\psscalebox{#1}{%
\newgame\notationoff\hidemoves{#2} 

\showboard
}}
\newcommand{\chessinvposnscale}[2]{%
\psscalebox{#1}{\newgame\notationoff\hidemoves{#2} 

\showinverseboard
}}

\newcommand{\chesssimulscale}[4]{\begin{psmatrix}[colsep=#4]\\[1ex]\rput(0,.1)
{\rnode{a}{\chessposnscale{#3}{#1}}}&\small$\to$&\rput(0,.1){\rnode{b}{\chessinvposnscale{#3}{#2}}}\\[0ex]\end{psmatrix}}
\newcommand{\invchesssimulscale}[4]{\begin{psmatrix}[colsep=#4]\\[1ex]\rput(0,.1)
{\rnode{c}{\chessinvposnscale{#3}{#1}}}&\small$\to$&\rput(0,.1){\rnode{d}{\chessposnscale{#3}{#2}}}\\[0ex]\end{psmatrix}}
\newcommand{\chesssimul}[2]{\chesssimulscale{#1}{#2}{0.2}{3.8ex}}
\newcommand{\oxsimulscale}[4]{\begin{psmatrix}[colsep=#4]\\[1ex]\rput(0,.1)
{\rnode{x}{\psscalebox{#3}{#1}}}&\small$\to$&\rput(0,.1){\rnode{y}{\psscalebox{#3}{#2}}}\\[0ex]\end{psmatrix}}
\newcommand{\oxsimul}[2]{\oxsimulscale{#1}{#2}{.8}{3.8ex}}

\newcommand{\import}{\,\mbox{\raisebox{-.3ex}{\LARGE$\cdot$}\,}}

\newcommand{\arena}{\mathsf{A}}

\newcommand{\chesssys}{\Delta_{\wmove{N}}}
\newcommand{\chessmov}{M_{\wmove{N}}}

\newcommand{\lambdax}{\pmb{\lambda}}

\newcommand{\dummy}{\star}

\newcommand{\allGG}{\forall G\,.\,G\to G}

\title{\vspace*{-5.5ex}\Large
\textbf{Hypergames and full completeness for system \textit{F}\\ (\sc rough draft)}
\author{\\[-3ex]
\large Dominic J.\ D.\ Hughes
\thanks{Visiting scholar, Computer Science Department, Stanford University, CA 94305.}
\\[1ex]
\normalsize Stanford University\\
\normalsize August 25, 2006}\date{}
}

\begin{document}\thispagestyle{empty}
\maketitle

\begin{abstract}
This paper reviews the fully complete \emph{hypergames} model of
system $F$, presented a decade ago in the author's thesis.
Instantiating type variables is modelled by allowing ``games as
moves''.  The uniformity of a quantified type variable $\forall X$ is
modelled by
\emph{copycat expansion}: $X$ represents an unkown game, a kind of black box,
so all the player can do is copy moves between a positive occurrence
and a negative occurrence of $X$.  

This presentation is based on slides for a talk entitled ``Hypergame
semantics: ten years later'' given at \emph{Games for Logic and
Programming Languages}, Seattle, August 2006.
\end{abstract}

\section{Introduction}

Zwicker's \textsl{\textsf{Hypergame}} \cite{Zwi87} is an alternating
two-player game: one player chooses any
alternating game $G$ which terminates\footnote{Every legal sequence of
moves is finite.} (\eg\ ``\textsf{\textbf{O}'s \& \textbf{X}'s}'' or
\textsf{Chess}\footnote{To ensure termination, assume a draw is forced upon a
threefold repetition of a position (a variant of a standard rule).}),
then play proceeds in $G$.\footnote{The question \mbox{``\textsl{Does
\textsf{Hypergame} terminate?}''}, the \emph{Hypergame paradox}, amounts to a hereditary form of
Russell's paradox, known as Mirimanoff's paradox
\cite{Mir17}: ``\textsl{Is the set of well-founded sets well-founded?}''.  (Each `paradox' is illusory, being merely
due to the lack of formal definition of ``game'' or ``set''.)}

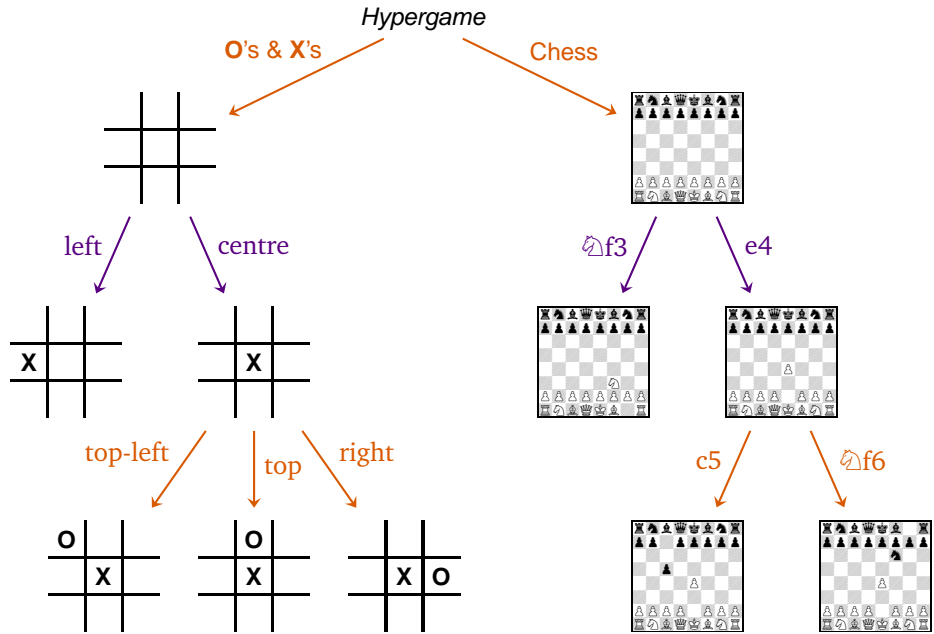
\begin{center}
\pspicture*(-14,-8.3)(15,.3)
\psset{nodesep=5pt,xunit=.25cm,yunit=.57cm}
\rput(0,0){\rnode{root}{\textsl{\textsf{Hypergame}}}}
\rput(14,-3){\rnode{chess}{\chessposn{}}}
\rput(-14,-3){\rnode{empty}{\oxgrid{}{}{}{}{}{}{}{}{}}}
\ncline[arrows=->,nodesepA=3pt,linecolor=ocolor]{root}{chess}
\aput(0.5){\ocolor\textsf{Chess}}
\ncline[arrows=->,nodesepA=3pt,linecolor=ocolor]{root}{empty}
\bput(0.5){\ocolor\textsf{\textbf{O}'s \& \textbf{X}'s}}
\rput(9,-8){\rnode{nf3}{\chessposn{1.Nf3}}}
\ncline[arrows=->,linecolor=pcolor]{chess}{nf3}
\bput(0.5){\pcolor\wmove{Nf3}}
\rput(19,-8){\rnode{e4}{\chessposn{1.e4}}}
\ncline[arrows=->,linecolor=pcolor]{chess}{e4}
\aput(0.5){\pcolor e4}
\rput(14,-13){\rnode{sicilian}{\chessposn{1.e4 c5}}}
\ncline[arrows=->,linecolor=ocolor]{e4}{sicilian}
\bput(0.5){\ocolor c5}
\rput(24,-13){\rnode{alekhine}{\chessposn{1.e4 Nf6}}}
\ncline[arrows=->,linecolor=ocolor]{e4}{alekhine}
\aput(0.5){\ocolor\wmove{Nf6}}
\rput(-9,-8){\rnode{centre}{\oxgrid{}{}{}{}{X}{}{}{}{}}}
\ncline[arrows=->,linecolor=pcolor]{empty}{centre}
\aput(0.5){\pcolor centre}
\rput(-19,-8){\rnode{left}{\oxgrid{}{}{}{X}{}{}{}{}{}}}
\ncline[arrows=->,linecolor=pcolor]{empty}{left}
\bput(0.5){\pcolor left}
\rput(-17,-13){\rnode{x}{\oxgrid{O}{}{}{}{X}{}{}{}{}}}
\ncline[arrows=->,linecolor=ocolor]{centre}{x}
\bput(0.5){\ocolor top-left}
\rput(-9,-13){\rnode{y}{\oxgrid{}{O}{}{}{X}{}{}{}{}}}
\ncline[arrows=->,linecolor=ocolor]{centre}{y}
\aput(0.5){\ocolor top} 
\rput(-1,-13){\rnode{z}{\oxgrid{}{}{}{}{X}{O}{}{}{}}}
\ncline[arrows=->,linecolor=ocolor]{centre}{z}
\aput(0.5){\ocolor right}
\endpspicture
\end{center}
At the \textit{Imperial College Games Workshop} in 1996, the author
illustrated how hypergames --- games in which games can be played as
moves --- can model languages with universal quantification.
Originally implemented in \cite{Hug97} for Girard's system $F$
\cite{Gir71,GLT89}, the idea is quite general, and has been successfully
applied to affine linear logic \cite{MO01,Mur01} and Curry-style type
isomorphisms \cite{deL06}.

\subsection{Universally quantified games}\label{univ-quant-intro}

Recall the little girl Anna-Louise who wins one point out of two in a
``simultaneous display'' against chess world champions Spassky and
Fischer \cite[Theorem~51]{Con76}.
She faces Spassky as Black and Fischer as White, and copies moves back
and forth, indirectly playing one champion against the other.  When
Spassky opens with the Queen's pawn \wmove{d4}, she opens \wmove{d4}
against Fischer; when Fischer responds with the King's knight
\wmove{Nf6}, she responds \wmove{Nf6} against Spassky, and so on.
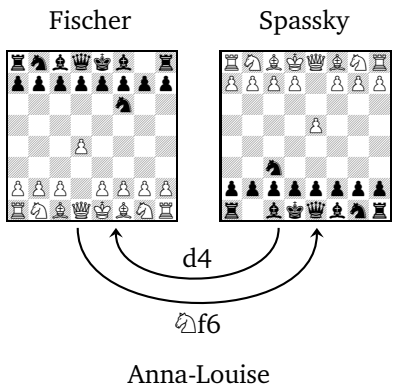
\begin{center}\vspace{2ex}
\begin{psmatrix}[colsep=5ex,rowsep=2ex]
Fischer && Spassky \\ \\
\rput(0,.1){\rnode{FischerBoard}{\psscalebox{.4}{%
\newgame
\notationoff
\hidemoves{1.d4 Nf6} 

\showboard
}}}
&&
\rput(0,.1){\rnode{SpasskyBoard}{\psscalebox{.4}{%
\newgame
\notationoff
\hidemoves{1.d4 Nf6} 

\showinverseboard
}}}
\end{psmatrix}
\nccurve[ncurv=.82,arrows=->,nodesep=.5ex,angleA=-90,angleB=-90,offsetA=-2ex,offsetB=-2ex]{SpasskyBoard}{FischerBoard}
\bput(0.5){\wmove{d4}}
\nccurve[ncurv=.83,arrows=->,nodesep=.5ex,angleA=-90,angleB=-90,offsetA=-1ex,offsetB=-1ex]{FischerBoard}{SpasskyBoard}
\bput(0.5){\wmove{Nf6}}
\vspace{17ex}

Anna-Louise
\vspace{3ex}\end{center}
We shall write $\simul{G}$ for such a simultaneous display with a game
$G$ (so Anna-Louise played the game $\simul{\chess}$ above, as second
player, against the Fischer-Spassky
team).\footnote{\label{backtracking}Conway writes $-G+G$, or $G-G$
\cite[Chapter\,7]{Con76}.  Later on, we shall add a form of
backtracking to our games so that Anna-Louise may restart the
game with Fischer as many times as she likes, corresponding to the
intuitionism of the arrow $\to$ of system $F$, in which a function may
read its argument any number of times \cite{Lor60,Fel85,Coq91,HO}.
To maintain the focus on universal quantification, here in the
introduction we shall ignore the availability backtracking.}

Observing that her copycat strategy is not specific to chess,
Anna-Louise declares that she will tackle the Fischer-Spassky team in
a more grandiose spectacle: she will give them an additional first
move, to decide the game for simultaneous display.  For example, the
Fischer-Spassky team might choose \textsf{Chess}, thereby opting for
the simultaneous display $\simul{\chess}$,
and play continues as above.
Or they might choose \ox{}, opting for the simultaneous
display $\simul{\ox}$, and open with \oxX{} in the centre of Spassky's
grid; Anna-Louise copies that $\oxX$ accross as her opening move on
Fischer's grid; Fischer responds with \oxO{} in (his) top-left; Anna copies
this $\oxO$ back to Spassky; and so on:
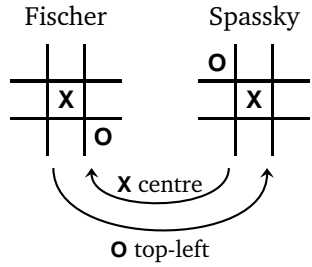
\begin{center}\vspace{2ex}\label{chess-chess}
\begin{psmatrix}[colsep=4ex,rowsep=0ex]
Fischer && Spassky  \\[4ex]
\rput(0,.1){\rnode{FischerBoard}{\psscalebox{1}{\oxgrid{}{}{}{}{X}{}{}{}{O}}}}
&&
\rput(0,.1){\rnode{SpasskyBoard}{\psscalebox{1}{\oxgrid{O}{}{}{}{X}{}{}{}{}}}}
\\[14ex]
&\makebox[0ex]{Anna-Louise}
\\[3ex]
\end{psmatrix}
\nccurve[arrows=->,nodesep=1ex,angleA=-90,angleB=-90,offsetA=-2ex,offsetB=-2ex]{SpasskyBoard}{FischerBoard}
\bput(0.5){\small\textsf{\textbf X} centre}
\nccurve[ncurv=.77,arrows=->,nodesep=1ex,angleA=-90,angleB=-90,offsetA=-1ex,offsetB=-1ex]{FischerBoard}{SpasskyBoard}
\bput(0.5){\small\textsf{\textbf O} top-left}
\end{center}
The key novelty of \cite{Hug97} was 
to define this as a formal game, a \emph{hypergame} or
\emph{universally quantified game},
which we shall write as
$$\allGG$$ 
The tree of $\allGG$ is illustrated below. Similar
in spirit to Zwicker's hypergame\footnote{The author was unaware of
Zwicker's work while preparing \cite{Hug97}, hence the lack of
reference to Zwicker in that paper, and in the author's thesis
\cite{Hug00}.}, it differs in the fact that the first player not
only chooses $G$ but also plays an opening move $m$ in $G$.  We call
such a compound move (importing a game, and playing a move in a game)
a \emph{hypermove}.
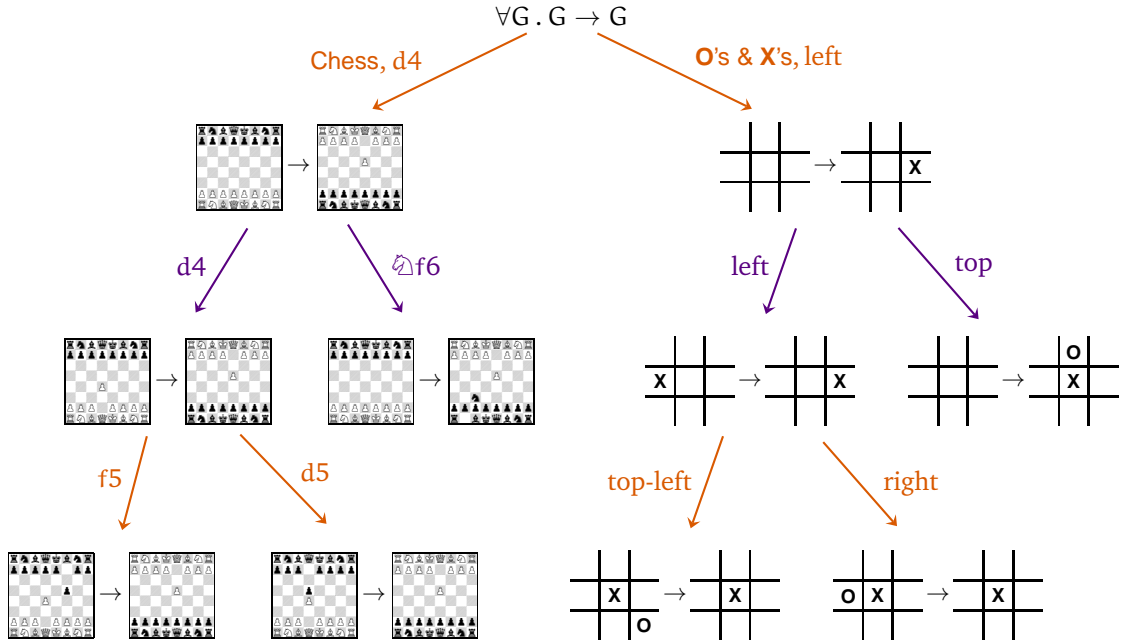
\begin{center}\label{Htree}
\pspicture*(-15,-8.3)(15,.3)
\psset{nodesep=5pt,xunit=.25cm,yunit=.57cm}
\rput(0,0){\rnode{root}{$\allGG$}}

\rput(-14,-3){\rnode{chess}{\chesssimul{}{1.d4}}}

\rput(14,-3){\rnode{empty}{\oxsimul{\oxempty}{\oxgrid{}{}{}{}{}{X}{}{}{}}}}
\ncline[arrows=->,nodesepA=3pt,linecolor=ocolor]{root}{chess}
\bput(0.8){\ocolor$\textsf{Chess},\wmove{d4}$}
\ncline[arrows=->,nodesepA=3pt,linecolor=ocolor]{root}{empty}
\aput(0.8){\ocolor$\textsf{\textbf{O}'s \& \textbf{X}'s},\text{left}$}
\rput(-21,-8){\rnode{d4}{\chesssimul{1.d4}{1.d4}}}
\ncline[nodesepA=-.5ex,nodesepB=-2.5ex,arrows=->,linecolor=pcolor]{chess}{d4}
\bput(0.6){\pcolor$\wmove{d4}$}
\rput(-7,-8){\rnode{nf6}{\chesssimul{}{1.d4 Nf6}}}
\ncline[nodesepA=-.8ex,nodesepB=-2.5ex,arrows=->,linecolor=pcolor]{chess}{nf6}
\aput(0.6){\pcolor$\wmove{Nf6}$}
\rput(-24,-13){\rnode{f5}{\chesssimul{1.d4 f5}{1.d4}}}
\ncline[nodesepA=-1ex,nodesepB=-3.7ex,arrows=->,linecolor=ocolor]{d4}{f5}
\bput(0.5){\ocolor$\wmove{f5}$}
\rput(-10,-13){\rnode{d5}{\chesssimul{1.d4 d5}{1.d4}}}
\ncline[nodesepA=1ex,nodesepB=-1ex,arrows=->,linecolor=ocolor]{d4}{d5}
\aput(0.6){\ocolor$d5$}
\rput(10,-8){\rnode{copy}{\oxsimul{\oxgrid{}{}{}{X}{}{}{}{}{}}{\oxgrid{}{}{}{}{}{X}{}{}{}}}}
\ncline[nodesepA=-.5ex,nodesepB=-3.5ex,arrows=->,linecolor=pcolor]{empty}{copy}
\bput(0.52){\pcolor left}
\rput(24,-8){\rnode{right}{\oxsimul{\oxempty}{\oxgrid{}{O}{}{}{X}{}{}{}{}}}}
\ncline[nodesepA=1ex,nodesepB=-2ex,arrows=->,linecolor=pcolor]{empty}{right}
\aput(0.6){\pcolor top}
\rput(6,-13){\rnode{p}{\oxsimul{\oxgrid{}{}{}{}{X}{}{}{}{O}}{\oxcentrex}}}
\ncline[nodesepA=-1ex,nodesepB=-3.5ex,arrows=->,linecolor=ocolor]{copy}{p}
\bput(0.6){\ocolor top-left} 
\rput(20,-13){\rnode{q}{\oxsimul{\oxgrid{}{}{}{O}{X}{}{}{}{}}{\oxcentrex}}}
\ncline[nodesepA=1ex,nodesepB=-1.5ex,arrows=->,linecolor=ocolor]{copy}{q}
\aput(0.7){\ocolor right}
\endpspicture
\end{center}

\subsection{Self-reference (without paradox)}
In the tree above, we have shown two cases for instantiating $G$ in
the hypergame $H\,=\,\allGG$, either to \chess{} or to
\ox.  But it is also possible to instantiate $G$ to a hypergame, or indeed, 
to $H$ itself.  We consider this case below.
The initial state is:
\begin{center}\vspace{1ex}
\begin{psmatrix}[colsep=-2ex,rowsep=0ex]
Fischer && Spassky  \\[2ex]
&$\allGG$
\\[2ex]
&\makebox[0ex]{Anna-Louise}
\\[.2ex]
\;
\end{psmatrix}
\end{center}
Fischer and Spassky begin by importing a game for $G$, in this case,
$H\,=\,\allGG$ itself, yielding a simultaneous display of
$H$:
\begin{center}\vspace{1ex}
\begin{psmatrix}[colsep=-2ex,rowsep=0ex]
Fischer && Spassky  \\[2ex]
&$H$\;\;{\Large$\to$}\;\;$H$
\\[2ex]
&\makebox[0ex]{Anna-Louise}
\\[.5ex]
\;
\end{psmatrix}
\end{center}
In other words, we have:
\begin{center}\vspace{2ex}
\begin{psmatrix}[colsep=4ex,rowsep=0ex]
Fischer && Spassky  \\[3ex]
\rput(0,.1){\rnode{FischerBoard}{$\forall G_1\,.\,G_1\to G_1$}}
&\Large\;\;\;\;\;$\to$\;\;\;\;\;&
\rput(0,.1){\rnode{SpasskyBoard}{$\forall G_2\,.\,G_2\to G_2$}}
\\[3ex]
&\makebox[0ex]{Anna-Louise}
\\[2ex]
\;
\end{psmatrix}
\end{center}
The local bound variable $G$ is renamed in each component to clarify
the evolution of the game below.\footnote{The scope of $\forall G_1$
in the diagram does not extend past the central arrow
{\normalsize$\to$}.  In other words, formally the game played by
Anna-Louise is $(\forall G_1.G_1\to G_1)\,\to\,(\forall
G_2.G_2\to G_2).$}
As in the simultaneous display $\chess\to\chess$, where Spassky opened
with a move on his chessboard, here in $H\to H$ Spassky must complete
the opening hypermove by playing a move on his copy of $H$.  Since
$H\,=\,\forall G_2\,.\,G_2\to G_2$ is a hypergame, opening $H$ requires
importing \emph{another} game, instantiating $G_2$.  Suppose he
chooses \chess{} for $G_2$:
\begin{center}\vspace{2ex}
\hspace*{4ex}\begin{psmatrix}[colsep=4ex,rowsep=0ex]
Fischer && Spassky  \\[6ex]
\rput(0,.1){\rnode{FischerBoard}{$\forall G_1\,.\,G_1\to G_1$}}
&\Large\hspace*{7ex}$\to$\hspace*{7ex}&
\hspace*{4ex}\chesssimulscale{}{}{.45}{8.5ex}
\\[9ex]
&\makebox[0ex]{Anna-Louise}
\\[2ex]
\;
\end{psmatrix}
\end{center}
Now Spassky has his own local simultaneous display $\chess\to\chess$.
To complete his opening (hyper)move on the overall game, he must open
this chess display.  Suppose he plays
\wmove{Nf3} (necessarily on the right board, where it is his turn since he has White):
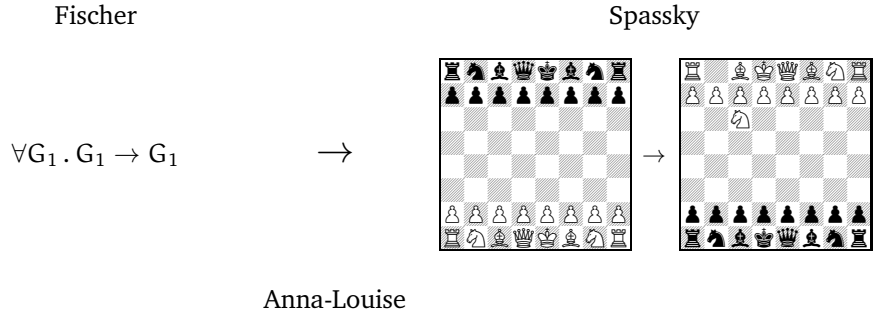
\begin{center}\vspace{2ex}
\hspace*{4ex}\begin{psmatrix}[colsep=4ex,rowsep=0ex]
Fischer && Spassky  \\[6ex]
\rput(0,.1){\rnode{FischerBoard}{$\forall G_1\,.\,G_1\to G_1$}}
&\Large\hspace*{7ex}$\to$\hspace*{7ex}&
\hspace*{4ex}\chesssimulscale{}{1.Nf3}{.45}{8.5ex}
\\[9ex]
&\makebox[0ex]{Anna-Louise}
\\[2ex]
\;
\end{psmatrix}
\end{center}
Now it is Anna-Louise's turn.  
She has three options: (1) respond to Spassky as Black on
the rightmost chess board, (2) respond to Spassky as White on the
other chess board, or (3) play an opening move against Fischer
in $\forall G_1\,.\,G_1\to G_1$.  
We consider the last case, since it is the most interesting.  Suppose
Anna-Louise chooses to import \ox{} for $G_1$:
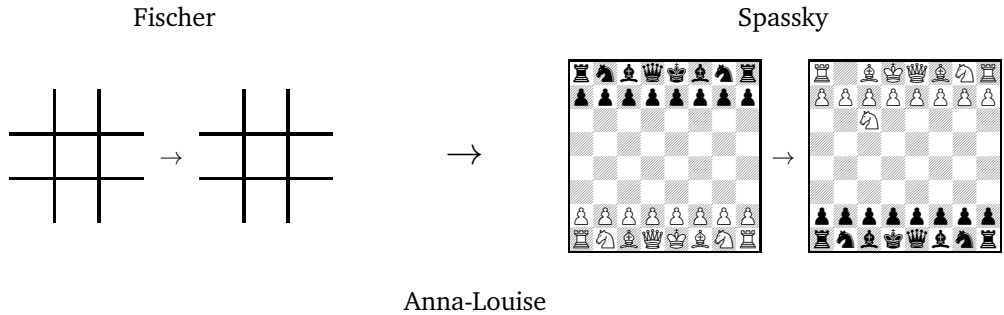
\begin{center}\vspace{2ex}
\hspace*{4ex}\begin{psmatrix}[colsep=4ex,rowsep=0ex]
Fischer\hspace*{3.5ex} && Spassky  \\[6ex]
\oxsimulscale{\oxempty}{\oxempty}{1.2}{6.5ex}\hspace*{4ex}
&\Large\hspace*{7ex}$\to$\hspace*{7ex}&
\hspace*{4ex}\chesssimulscale{}{1.Nf3}{.45}{8.5ex}
\\[9ex]
&\makebox[0ex]{Anna-Louise}\;\;\;\;
\\[3ex]
\;
\end{psmatrix}\;\;\;\;\;\;\;\;
\end{center}
Now Fischer has his own local simultaneous display $\ox\to\ox$.
For Anna-Louise to complete her hypermove, she must play a move on
$\ox\to\ox$ (necessarily in the right of the two grids, the one in
which it her turn).
Suppose she plays her $\oxX$ top-right:
\begin{center}\vspace{2ex}
\hspace*{4ex}\begin{psmatrix}[colsep=4ex,rowsep=0ex]
Fischer\hspace*{3.5ex} && Spassky  \\[6ex]
\oxsimulscale{\oxempty}{\oxgrid{}{}{X}{}{}{}{}{}{}}{1.2}{6.5ex}\hspace*{4ex}
&\Large\hspace*{7ex}$\to$\hspace*{7ex}&
\hspace*{4ex}\chesssimulscale{}{1.Nf3}{.45}{8.5ex}
\\[9ex]
&\makebox[0ex]{Anna-Louise}\;\;\;\;
\\[3ex]
\;
\end{psmatrix}\;\;\;\;\;\;\;\;
\end{center}
Fischer responds either with an $\oxO$ in the same grid, or with an
$\oxX$ in the empty grid, and play continuous in the two local
simultaneous displays, $\ox\to\ox$ against Fischer and
$\chess\to\chess$ against Spassky.

But to remain consistent with her copycat strategy, Anna-Louise must
mimic Spassky.  Instead of importing \ox{} for $G_1$ against Fischer,
she must import \chess{} and open with the White move
\wmove{Nf3}, exactly as Spassky did:
\begin{center}\vspace{2ex}
\hspace*{4ex}\begin{psmatrix}[colsep=4ex,rowsep=0ex]
Fischer\hspace*{3.5ex} && Spassky  \\[6ex]
\invchesssimulscale{}{1.Nf3}{.45}{8.5ex}\hspace*{4ex}
&\Large\hspace*{7ex}$\to$\hspace*{7ex}&
\hspace*{4ex}\chesssimulscale{}{1.Nf3}{.45}{8.5ex}
\\[15ex]
&\makebox[0ex]{Anna-Louise}\;\;\;\;
\\[2ex]
\;
\end{psmatrix}\;\;\;\;\;\;\;\;%
\nccurve[ncurv=.35,arrows=<-,nodesep=.5ex,angleA=-60,angleB=-120]{d}{b}\aput(0.5){$\chess$\:,\:\wmove{Nf3}}
\end{center}
Fischer might now open his other board with \wmove{e4}, which
Anna-Louise would copy back to the corresponding board against
Spassky:
\begin{center}\vspace{2ex}
\hspace*{4ex}\begin{psmatrix}[colsep=4ex,rowsep=0ex]
Fischer\hspace*{3.5ex} && Spassky  \\[6ex]
\invchesssimulscale{1.e4}{1.Nf3}{.45}{8.5ex}\hspace*{4ex}
&\Large\hspace*{7ex}$\to$\hspace*{7ex}&
\hspace*{4ex}\chesssimulscale{1.e4}{1.Nf3}{.45}{8.5ex}
\\[14ex]
&\makebox[0ex]{Anna-Louise}\;\;\;\;
\\[1ex]
\;
\end{psmatrix}\;\;\;\;\;\;\;\;%
\nccurve[ncurv=.35,arrows=->,nodesep=.5ex,angleA=-60,angleB=-120]{c}{a}\aput(0.5){\wmove{e4}}%
\end{center}
Or perhaps Fischer responds with Black in the rightmost of his pair
of boards, with \wmove{d5}, which Anna-Louise copies to Spassky:
\begin{center}\vspace{2ex}
\hspace*{4ex}\begin{psmatrix}[colsep=4ex,rowsep=0ex]
Fischer\hspace*{3.5ex} && Spassky  \\[6ex]
\invchesssimulscale{}{1.Nf3 d5}{.45}{8.5ex}\hspace*{4ex}
&\Large\hspace*{7ex}$\to$\hspace*{7ex}&
\hspace*{4ex}\chesssimulscale{}{1.Nf3 d5}{.45}{8.5ex}
\\[14ex]
&\makebox[0ex]{Anna-Louise}\;\;\;\;
\\[1ex]
\;
\end{psmatrix}\;\;\;\;\;\;\;\;%
\nccurve[ncurv=.35,arrows=->,nodesep=.5ex,angleA=-60,angleB=-120]{d}{b}\aput(0.5){\wmove{d5}}%
\end{center}
Either way, she continues to copy moves between the four boards
according to the following geometry of copycat links:
\begin{center}\vspace{1ex}
\hspace*{4ex}\begin{psmatrix}[colsep=4ex,rowsep=0ex]
Fischer\hspace*{3.5ex} && Spassky  \\[6ex]
\invchesssimulscale{}{1.Nf3}{.45}{8.5ex}\hspace*{4ex}
&\Large\hspace*{7ex}$\to$\hspace*{7ex}&
\hspace*{4ex}\chesssimulscale{}{1.Nf3}{.45}{8.5ex}
\\[14ex]
&\makebox[0ex]{Anna-Louise}\;\;\;\;
\\[1ex]
\;
\end{psmatrix}\;\;\;\;\;\;\;\;%
\nccurve[ncurv=.35,arrows=<->,nodesep=.5ex,angleA=-60,angleB=-120]{c}{a}%
\nccurve[ncurv=.35,arrows=<->,nodesep=.5ex,angleA=-60,angleB=-120]{d}{b}%
\end{center}
This copycat strategy corresponds to the polymorphic identity system
$F$ term $$\Lambda G.\lambda g^G.g$$ of type $\forall G\,.\,G\to G\,$.

\subsection{Uniformity}

Consider again the original Fischer-Spassky simultaneous display, with
chess.  Add Kasparov to the team, playing Black.
\begin{center}\vspace{2ex}
\begin{psmatrix}[colsep=5ex,rowsep=2ex]
Kasparov && Fischer && Spassky  \\ \\
\rput(0,.1){\rnode{KasparovBoard}{\psscalebox{.45}{%
\newgame
\notationoff

\showboard
}}}\,
&&
\rput(0,.1){\rnode{FischerBoard}{\psscalebox{.45}{%
\newgame
\notationoff

\showboard
}}}\;
&&
\rput(0,.1){\rnode{SpasskyBoard}{\psscalebox{.45}{%
\newgame
\notationoff

\showinverseboard
}}}%
\end{psmatrix}
\vspace{9ex}

Anna-Louise
\vspace{3ex}\end{center}
Anna-Louise has two distinct ways to guarantee picking up a point.
Either she copies moves between Spassky and Fischer, as before, while
ignoring Kasparov (never playing a move against him),
\begin{center}\vspace{2ex}
\begin{psmatrix}[colsep=5ex,rowsep=2ex]
Kasparov && Fischer && Spassky  \\ \\
\rput(0,.1){\rnode{KasparovBoard}{\psscalebox{.45}{%
\newgame
\notationoff

\showboard
}}}\,
&&
\rput(0,.1){\rnode{FischerBoard}{\psscalebox{.45}{%
\newgame
\notationoff
\hidemoves{1.d4 Nf6} 

\showboard
}}}\;
&&
\rput(0,.1){\rnode{SpasskyBoard}{\psscalebox{.45}{%
\newgame
\notationoff
\hidemoves{1.d4 Nf6} 

\showinverseboard
}}}
\\[16ex]
&&\makebox[0ex][c]{Anna-Louise}
\\[.5ex]
\;
\end{psmatrix}
\nccurve[ncurv=.82,arrows=->,nodesep=.5ex,angleA=-90,angleB=-90,offsetA=-2ex,offsetB=-2ex]{SpasskyBoard}{FischerBoard}
\bput(0.5){\wmove{d4}}
\nccurve[ncurv=.83,arrows=->,nodesep=.5ex,angleA=-90,angleB=-90,offsetA=-1ex,offsetB=-1ex]{FischerBoard}{SpasskyBoard}
\bput(0.5){\wmove{Nf6}}
\end{center}
or she copies moves between Spassky and Kasparov, ignoring Fischer:
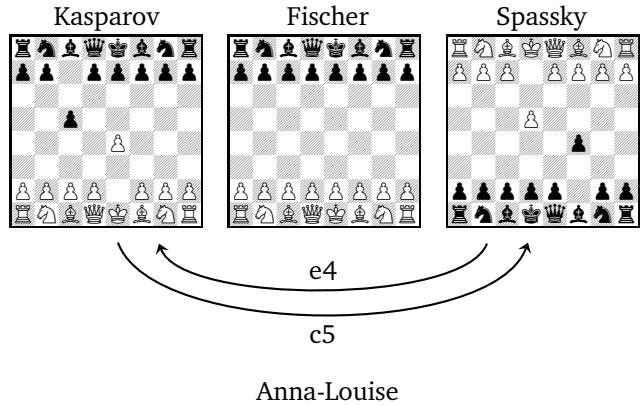
\begin{center}\vspace{2ex}
\begin{psmatrix}[colsep=5ex,rowsep=2ex]
Kasparov && Fischer && Spassky  \\ \\
\rput(0,.1){\rnode{KasparovBoard}{\psscalebox{.45}{%
\newgame
\notationoff
\hidemoves{1.e4 c5} 

\showboard
}}}\,
&&
\rput(0,.1){\rnode{FischerBoard}{\psscalebox{.45}{%
\newgame
\notationoff

\showboard
}}}\;
&&
\rput(0,.1){\rnode{SpasskyBoard}{\psscalebox{.45}{%
\newgame
\notationoff
\hidemoves{1.e4 c5} 

\showinverseboard
}}}
\\[16ex]
&&\makebox[0ex][c]{Anna-Louise}
\\[.5ex]
\;
\end{psmatrix}
\nccurve[ncurv=.4,arrows=->,nodesep=1.6ex,angleA=-110,angleB=-70,offsetA=-1ex,offsetB=-1ex]{SpasskyBoard}{KasparovBoard}
\bput(0.5){\wmove{e4}}
\nccurve[ncurv=.5,arrows=->,nodesep=.5ex,angleA=-70,angleB=-110,offsetA=-2ex,offsetB=-2ex]{KasparovBoard}{SpasskyBoard}
\bput(0.5){\wmove{c5}}
\end{center}
We shall write this triple simultaneous display as
$\chess\,\to\,\chess\,\to\,\chess$, and more generally, for any game
$G$, as $G\to G\to G$.\footnote{Again with the backtracking caveat: see footnote~\ref{backtracking}.}

Now consider the universally quantified form of this game, the hypergame
$$\forall G\,.\,G\to G\to G.$$
As with $\allGG$ discussed above, the
Kasparov-Fischer-Spassky team, KFS, now has the right to choose the
game of the triple simultaneous display, as part of their opening
(hyper)move.
We shall say that Anna-Louise's strategy is \defn{uniform} in this
setting if
\begin{itemize}
\item irrespective of the game chosen by KFS, she always ignores the same player, Kasparov or Fischer.
\end{itemize}
Otherwise her strategy is \defn{ad hoc}.  For example, her strategy
would be ad hoc if, when KFS chooses \chess{}, she ignores Kasparov
and copies chess moves between Fischer and Spassky, but when KFS
chooses \ox{}, she ignores Fischer and copies \oxX{} and \oxO{} moves
between Kasparov and Spassky.
In this case the geometry of her move copying depends on the game
imported by FKS: she is not treating $G$ as a ``black box''.

There are only two uniform strategies for Anna-Louise: either she
always copies between \emph{Kasparov} and Spassky, ignoring Fischer, or she always copies
between \emph{Fischer} and Spassky, ignoring Kasparov.
These correspond to the system $F$ terms 
\begin{displaymath}
\begin{array}{c}
\Lambda G\,.\:\lambda k^G\,.\,\lambda f^G\,.\,k\\[1ex]
\Lambda G\,.\:\lambda k^G\,.\,\lambda f^G\,.\,f
\end{array}
\end{displaymath}
respectively, of type 
$$\forall G\,.\,G\to G\to G\,,$$
where the variable $k$ corresponds to Kasparov and $f$ corresponds to Fischer.

More generally, with multiple bound $\forall$ variables and more
complicated game imports, we shall take
uniformity to mean that the links Anna-Louise sets up between
components (such as the Kasparov$\,\leftrightarrow\,$Spassky or
Fischer$\,\leftrightarrow\,$Spassky links above) must be independent
of the games imported by the opposing team: these imported games are impenetrable ``black boxes''.

\paragraph*{Fixed links.} 
Uniformity as independence from the particular games imported by the
opposing team will include independence from the not only the
\emph{identity} of those games, but also from their \emph{state}.
This will ensure that the geometry of Anna-Louise's copycat play
remains constant over time:
once she has committed to linking one component to another, she must
stick with that link for the rest of the hypergame.
To illustrate this aspect of uniformity, consider the quadruple chess
simultaneous display with Kasparov and Fischer playing Black, and
Karpov and Spassky playing White:
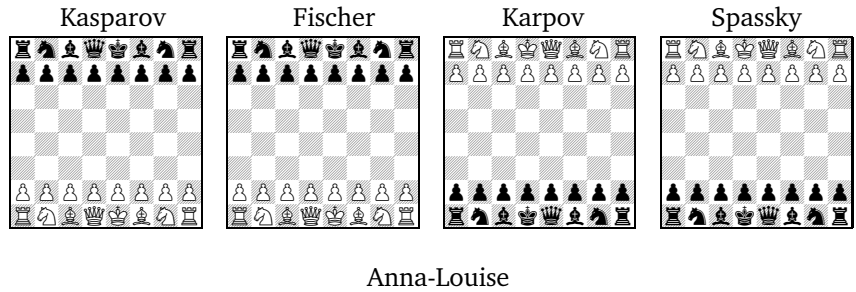
\begin{center}\vspace{2ex}
\begin{psmatrix}[colsep=5ex,rowsep=2ex]
Kasparov && Fischer && Karpov && Spassky  \\ \\
\rput(0,.1){\rnode{KasparovBoard}{\psscalebox{.45}{%
\newgame
\notationoff

\showboard
}}}\;\:
&&
\rput(0,.1){\rnode{FischerBoard}{\psscalebox{.45}{%
\newgame
\notationoff

\showboard
}}}\;\;\;
&&
\rput(0,.1){\rnode{KarpovBoard}{\psscalebox{.45}{%
\newgame
\notationoff

\showinverseboard
}}}\;
&&
\rput(0,.1){\rnode{SpasskyBoard}{\psscalebox{.45}{%
\newgame
\notationoff

\showinverseboard
}}}
\\[7ex]
&&&\makebox[0ex]{Anna-Louise}
\\[.5ex]
\;
\end{psmatrix}
\end{center}
We shall write $\chess\times\chess\,\to\,\chess\times\chess$ for this
simultaneous display.\footnote{With the backtracking caveat: see footnote~\ref{backtracking}.}
Suppose Spassky begins with \wmove{e4}.  Anna-Louise, playing copycat,
has a choice between copying this move to Fischer or to Kasparov.
Suppose she copies it to Fischer, who responds with \wmove{c5}, which
she duly copies back to Spassky:
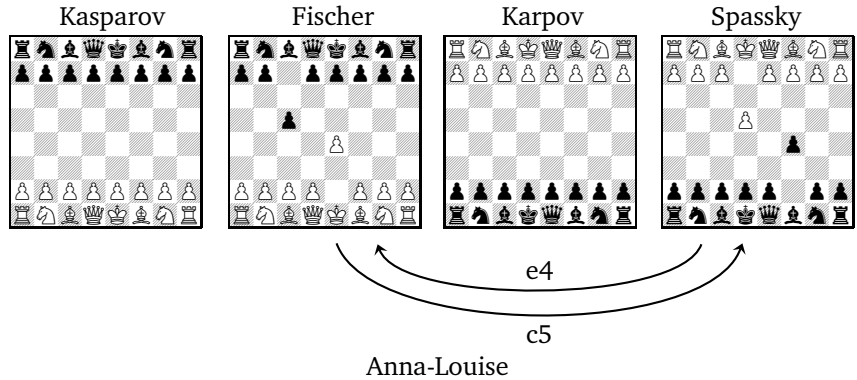
\begin{center}\vspace{2ex}
\begin{psmatrix}[colsep=5ex,rowsep=2ex]
Kasparov && Fischer && Karpov && Spassky  \\ \\
\rput(0,.1){\rnode{KasparovBoard}{\psscalebox{.45}{%
\newgame
\notationoff

\showboard
}}}\;\:
&&
\rput(0,.1){\rnode{FischerBoard}{\psscalebox{.45}{%
\newgame
\notationoff
\hidemoves{1.e4 c5}  

\showboard
}}}\;\;\:
&&
\rput(0,.1){\rnode{KarpovBoard}{\psscalebox{.45}{%
\newgame
\notationoff

\showinverseboard
}}}\,
&&
\rput(0,.1){\rnode{SpasskyBoard}{\psscalebox{.45}{%
\newgame
\notationoff
\hidemoves{1.e4 c5} 

\showinverseboard
}}}
\\[14ex]
&&&\makebox[0ex]{Anna-Louise}
\\[.5ex]
\;
\end{psmatrix}
\nccurve[ncurv=.4,arrows=->,nodesep=1.6ex,angleA=-110,angleB=-70,offsetA=-1ex,offsetB=-1ex]{SpasskyBoard}{FischerBoard}
\bput(0.5){\wmove{e4}}
\nccurve[ncurv=.5,arrows=->,nodesep=.5ex,angleA=-70,angleB=-110,offsetA=-2ex,offsetB=-2ex]{FischerBoard}{SpasskyBoard}
\bput(0.5){\wmove{c5}}
\end{center}
Suppose Karpov opens his game with the very same move as Spassky,
\wmove{e4}, which Anna-Louise copies accross to Kasparov (the only
destination where this move makes sense):
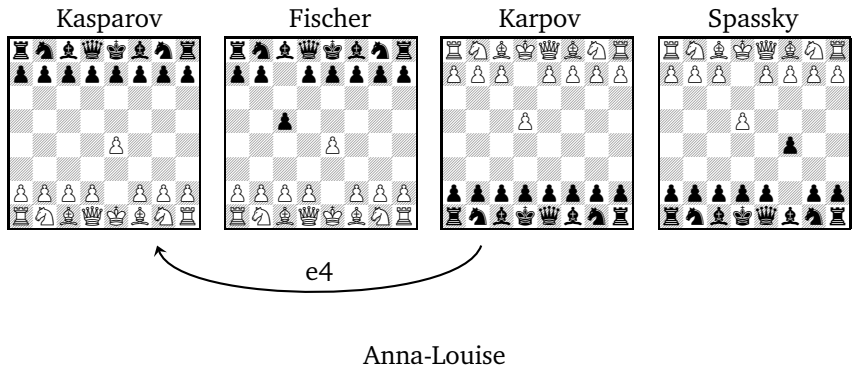
\begin{center}\vspace{2ex}
\begin{psmatrix}[colsep=5ex,rowsep=2ex]
Kasparov && Fischer && Karpov && Spassky  \\ \\
\rput(0,.1){\rnode{KasparovBoard}{\psscalebox{.45}{%
\newgame
\notationoff
\hidemoves{1.e4} 

\showboard
}}}\;\:
&&
\rput(0,.1){\rnode{FischerBoard}{\psscalebox{.45}{%
\newgame
\notationoff
\hidemoves{1.e4 c5}  

\showboard
}}}\;\;\;
&&
\rput(0,.1){\rnode{KarpovBoard}{\psscalebox{.45}{%
\newgame
\notationoff
\hidemoves{1.e4} 

\showinverseboard
}}}\;
&&
\rput(0,.1){\rnode{SpasskyBoard}{\psscalebox{.45}{%
\newgame
\notationoff
\hidemoves{1.e4 c5} 

\showinverseboard
}}}
\\[13ex]
&&&\makebox[0ex]{Anna-Louise}
\\[.5ex]
\;
\end{psmatrix}
\nccurve[ncurv=.4,arrows=->,nodesep=1.6ex,angleA=-110,angleB=-70,offsetA=-1ex,offsetB=-1ex]{KarpovBoard}{KasparovBoard}
\bput(0.5){\wmove{e4}}
\end{center}
Kasparov responds with the same move as Fischer, \wmove{c5}, which Anna-Louise copies back to Karpov:
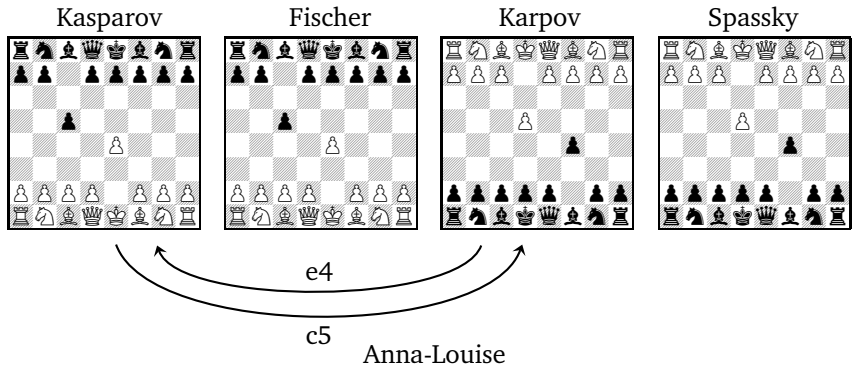
\begin{center}\vspace{2ex}
\begin{psmatrix}[colsep=5ex,rowsep=2ex]
Kasparov && Fischer && Karpov && Spassky  \\ \\
\rput(0,.1){\rnode{KasparovBoard}{\psscalebox{.45}{%
\newgame
\notationoff
\hidemoves{1.e4 c5} 

\showboard
}}}\;\:
&&
\rput(0,.1){\rnode{FischerBoard}{\psscalebox{.45}{%
\newgame
\notationoff
\hidemoves{1.e4 c5}  

\showboard
}}}\;\;\;
&&
\rput(0,.1){\rnode{KarpovBoard}{\psscalebox{.45}{%
\newgame
\notationoff
\hidemoves{1.e4 c5} 

\showinverseboard
}}}\;
&&
\rput(0,.1){\rnode{SpasskyBoard}{\psscalebox{.45}{%
\newgame
\notationoff
\hidemoves{1.e4 c5} 

\showinverseboard
}}}
\\[13ex]
&&&\makebox[0ex]{Anna-Louise}
\\[.5ex]
\;
\end{psmatrix}
\nccurve[ncurv=.4,arrows=->,nodesep=1.6ex,angleA=-110,angleB=-70,offsetA=-1ex,offsetB=-1ex]{KarpovBoard}{KasparovBoard}
\bput(0.5){\wmove{e4}}
\nccurve[ncurv=.5,arrows=->,nodesep=.5ex,angleA=-70,angleB=-110,offsetA=-2ex,offsetB=-2ex]{KasparovBoard}{KarpovBoard}
\bput(0.5){\wmove{c5}}
\end{center}
So far, Anna-Louise has linked Spassky with Fischer, and Karpov with Kasparov:
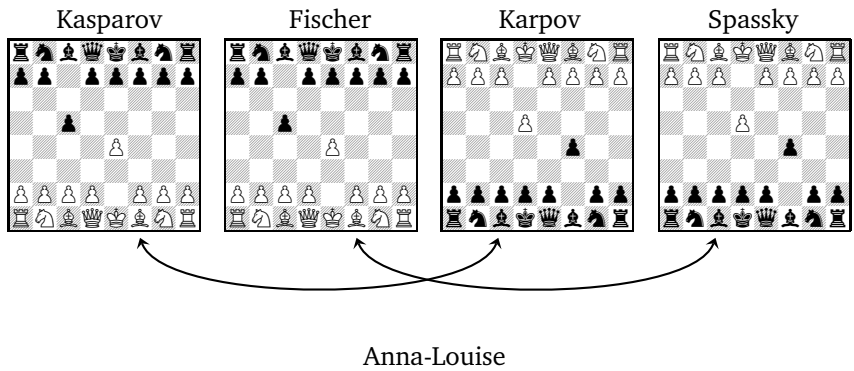
\begin{center}\vspace{2ex}
\begin{psmatrix}[colsep=5ex,rowsep=2ex]
Kasparov && Fischer && Karpov && Spassky  \\ \\
\rput(0,.1){\rnode{KasparovBoard}{\psscalebox{.45}{%
\newgame
\notationoff
\hidemoves{1.e4 c5} 

\showboard
}}}\;\:
&&
\rput(0,.1){\rnode{FischerBoard}{\psscalebox{.45}{%
\newgame
\notationoff
\hidemoves{1.e4 c5}  

\showboard
}}}\;\;\;
&&
\rput(0,.1){\rnode{KarpovBoard}{\psscalebox{.45}{%
\newgame
\notationoff
\hidemoves{1.e4 c5} 

\showinverseboard
}}}\;
&&
\rput(0,.1){\rnode{SpasskyBoard}{\psscalebox{.45}{%
\newgame
\notationoff
\hidemoves{1.e4 c5} 

\showinverseboard
}}}
\\[13ex]
&&&\makebox[0ex]{Anna-Louise}
\\[.5ex]
\;
\end{psmatrix}
\nccurve[ncurv=.4,arrows=<->,nodesep=.5ex,angleA=-70,angleB=-110]{KasparovBoard}{KarpovBoard}
\nccurve[ncurv=.4,arrows=<->,nodesep=.5ex,angleA=-70,angleB=-110]{FischerBoard}{SpasskyBoard}
\end{center}
By (contrived) coincidence, both pairs of linked boards happen to have
reached exactly the same state.  Therefore from this point onwards,
Anna-Louise could change the linkage, linking Kasparov with Spassky,
and Karpov with Fischer:
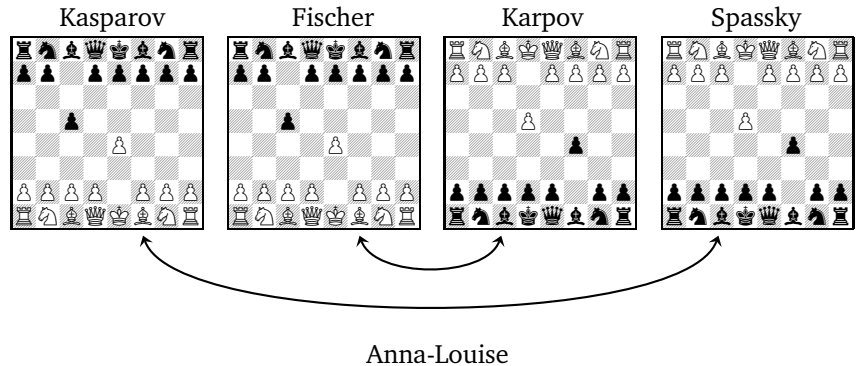
\begin{center}\vspace{2ex}
\begin{psmatrix}[colsep=5ex,rowsep=2ex]
Kasparov && Fischer && Karpov && Spassky  \\ \\
\rput(0,.1){\rnode{KasparovBoard}{\psscalebox{.45}{%
\newgame
\notationoff
\hidemoves{1.e4 c5} 

\showboard
}}}\;\:
&&
\rput(0,.1){\rnode{FischerBoard}{\psscalebox{.45}{%
\newgame
\notationoff
\hidemoves{1.e4 c5}  

\showboard
}}}\;\;\;
&&
\rput(0,.1){\rnode{KarpovBoard}{\psscalebox{.45}{%
\newgame
\notationoff
\hidemoves{1.e4 c5} 

\showinverseboard
}}}\;
&&
\rput(0,.1){\rnode{SpasskyBoard}{\psscalebox{.45}{%
\newgame
\notationoff
\hidemoves{1.e4 c5} 

\showinverseboard
}}}
\\[13ex]
&&&\makebox[0ex]{Anna-Louise}
\\[.5ex]
\;
\end{psmatrix}
\nccurve[ncurv=.35,arrows=<->,nodesep=.5ex,angleA=-70,angleB=-110]{KasparovBoard}{SpasskyBoard}
\nccurve[ncurv=.65,arrows=<->,nodesep=.5ex,angleA=-70,angleB=-110]{FischerBoard}{KarpovBoard}
\end{center}
For example, should Karpov respond with \wmove{Nf3}, she would copy
that move across to Fischer, then continue copying between Fischer and
Karpov, and between Kasparov and Spassky.

She could do this ``relinking'' for any game $G$, not just \chess, on
$G\times G\,\to\,G\times G$: no matter what the game $G$ is, she could
link the first and third $G$, and link the second and fourth $G$, but
if a point is reached in which all four copies of $G$ have the same
state, she switches the linkage, as in the chess example above.
If she consistently does this for all $G$, she has a strategy on the
hypergame $\forall G\,.\,G\times G\to G\times G$ which, in some fashion,
does not depend on $G$.  Such ``relinking'' strategies do not
correspond to system $F$ terms, and are eliminated from the model by
our uniformity condition: independence from $G$ means
independence not only from the \emph{identity} of $G$, but also from the
\emph{state} of $G$.

\subsection{Negative quantifiers}\label{sec-neg}

Linear polymorphism was modelled in \cite{Abr97} using a universal
notion of the games in \cite{AJ94,AJM00}.  Full completeness failed
for types with negative quantifiers.  In this subsection we illustrate
how the hypergames model successfully treats negative quantifiers.

The polarity of a quantifier in a type is positive or negative
according to the number of times it is to the left of an arrow (in the
syntactic parse tree of the type): positive if even, negative if odd.
For example, $\forall X$ is positive in $\forall X.T$ and $\forall
Y.\forall X.T$, negative in $(\forall X.U)\to T$ and $(V\to \forall
X.U)\to T$, and positive in $\big((\forall X.U)\to V\big)\to T$.

Consider the simultaneous display $H\to H$ where $H$ is the hypergame
$\allGG$:
\begin{center}\vspace{3ex}
\begin{psmatrix}[colsep=4ex,rowsep=0ex]
Fischer && Spassky  \\[4ex]
\rput(0,.1){\rnode{FischerBoard}{$\forall G_1\,.\,G_1\to G_1$}}
&\Large\;\;\;\;\;\;$\to$\;\;\;\;\;\;&
\rput(0,.1){\rnode{SpasskyBoard}{$\forall G_2\,.\,G_2\to G_2$}}
\\[3ex]
&\makebox[0ex]{Anna-Louise}
\\[.5ex]
\;
\end{psmatrix}
\end{center}
Fischer's quantifier $\forall G_1$ is negative.\footnote{The scope of
$\forall G_1$ in the diagram does not extend past the central arrow
{\normalsize$\to$}.}		
To kick off, Spassky must open the game $\forall G_2\,.\,G_2\to G_2$ in
front of him.  This is a hypergame, universally quantified, so he must
begin by instantiating $G_2$.  He chooses $G_2=\chess$, and opens
$\wmove{Nf3}$ on the board where he has White:
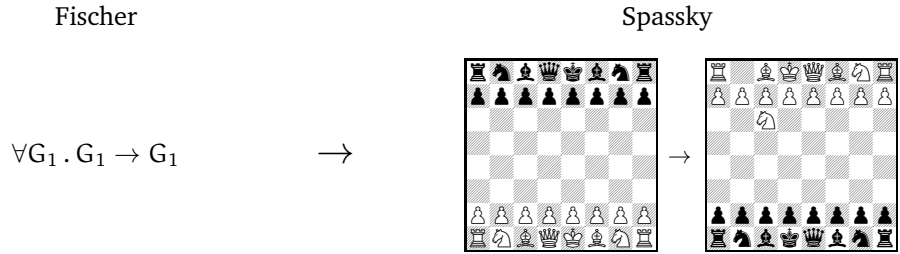
\begin{center}\vspace{2ex}
\hspace*{4ex}\begin{psmatrix}[colsep=4ex,rowsep=0ex]
Fischer && Spassky  \\[6ex]
\rput(0,.1){\rnode{FischerBoard}{$\forall G_1\,.\,G_1\to G_1$}}
&\Large\hspace*{7ex}$\to$\hspace*{7ex}&
\hspace*{6ex}\chesssimulscale{}{1.Nf3}{.45}{8.5ex}
\\[9ex]
&\makebox[0ex]{Anna-Louise}
\\[3ex]
\;
\end{psmatrix}
\end{center}
We shall consider three of the copycat strategies available to Anna-Louise from this point:
\begin{center}
\begin{tabular}{|c|l|l|}\hline
\bf \!\!Strategy\!\! & \bf Anna Louise\ldots & \bf \rule{0ex}{2.5ex}Corresponding term of type $H\to H$\raisebox{-1ex}{\strut} \\\hline
$\iota$ & 
\parbox{2.7in}{\rule{0ex}{2.2ex}\raggedright\ldots\ copies what Spassky did accross to Fischer: import $\chess$ and play $\wmove{Nf3}$\raisebox{-1ex}{\strut}} &
$\lambda h^H.h$ \\\hline
$\sigma$ & 
\parbox{2.7in}{\rule{0ex}{2.2ex}\raggedright\ldots\ plays copycat in Spassky's local chess display, ``playing Spassky against himself''\raisebox{-1ex}{\strut}} & 
$\lambda h^H.\Lambda G.\lambda g^{G}.g$ \\\hline
$\tau$ &
\parbox{2.7in}{\rule{0ex}{2.2ex}\raggedright\ldots\ imports $G_1\,=\,\chess\to\chess$ against Fischer,
 then copies moves between the six resulting boards, along three
 ``copycat links''\raisebox{-1ex}{\strut}} & 
$\lambda h^H.\Lambda G.\lambda g^{G}.h_{G\to G}(\lambda x^{G}.x)g$ \\\hline
\end{tabular}
\end{center}
The notation $h_U$ in the third term denotes the application of $h$ to
the type $U$.

\paragraph*{The first copycat strategy $\iota\,$.}
Anna-Louise opens the hypergame $\forall G_1\,.\,G_1\to G_1$ in front
of Fischer by mimicking Spassky: she imports $\chess$ for $G_1$ and
opens with $\wmove{Nf3}$ as White:
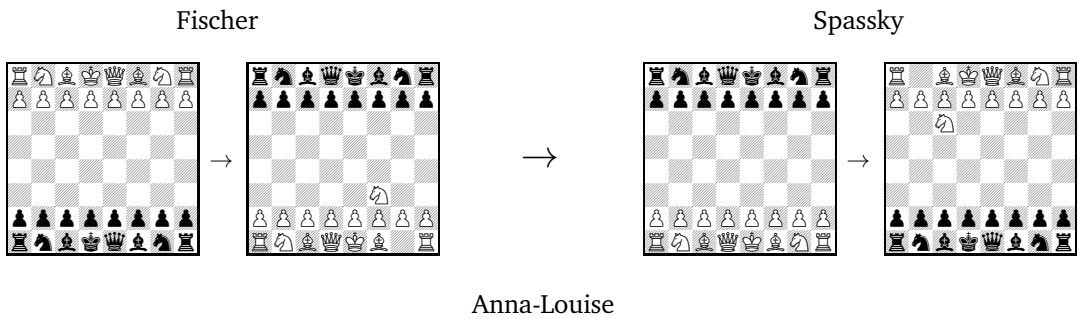
\begin{center}\vspace{2ex}
\begin{psmatrix}[colsep=4ex,rowsep=0ex]
Fischer\;\;\;\;\;\;\;\;&&Spassky\\[6ex]
\invchesssimulscale{}{1.Nf3}{.45}{8.5ex}\hspace*{4ex}
&\Large\hspace*{7ex}$\to$\hspace*{7ex}&
\hspace*{4ex}\chesssimulscale{}{1.Nf3}{.45}{8.5ex}
\\[9ex]
&\makebox[0ex]{Anna-Louise}\;\;\;\;\;\;
\\[2ex]
\;
\end{psmatrix}
\end{center}
She then copies moves between the four boards according to the
following geometry of copycat links:
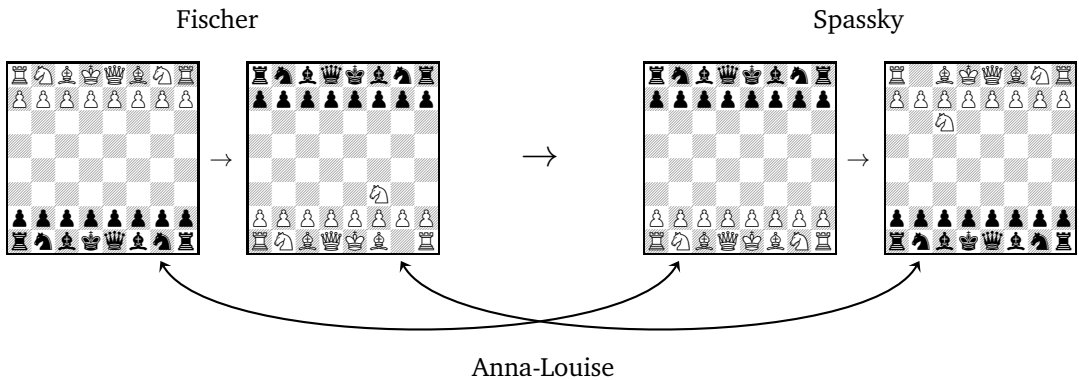
\begin{center}\vspace{2ex}
\begin{psmatrix}[colsep=4ex,rowsep=0ex]
Fischer\;\;\;\;\;\;\;\;&&Spassky\\[6ex]
\invchesssimulscale{}{1.Nf3}{.45}{8.5ex}\hspace*{4ex}
&\Large\hspace*{7ex}$\to$\hspace*{7ex}&
\hspace*{4ex}\chesssimulscale{}{1.Nf3}{.45}{8.5ex}
\\[14ex]
&\makebox[0ex]{Anna-Louise}\;\;\;\;\;\;
\\[2ex]
\;
\end{psmatrix}%
\nccurve[ncurv=.4,arrows=<->,nodesep=.5ex,angleA=-60,angleB=-120]{c}{a}
\nccurve[ncurv=.4,arrows=<->,nodesep=.5ex,angleA=-60,angleB=-120]{d}{b}
\end{center}
This copycat strategy $\iota$ corresponds to the identity system $F$
term $$\lambda h^H.h$$ of type $H\to H$.  (Recall $H\,=\,\forall G\,.\,G\to
G$.) The same strategy models the $\eta$-expanded variant $\lambda
h^H.\Lambda G.\lambda g^G.h_G\,g$.

\paragraph*{The second copycat strategy $\sigma\,$.}
The second copycat strategy $\sigma$ ``plays Spassky against
himself''.  Recall the state after Spassky's opening move:
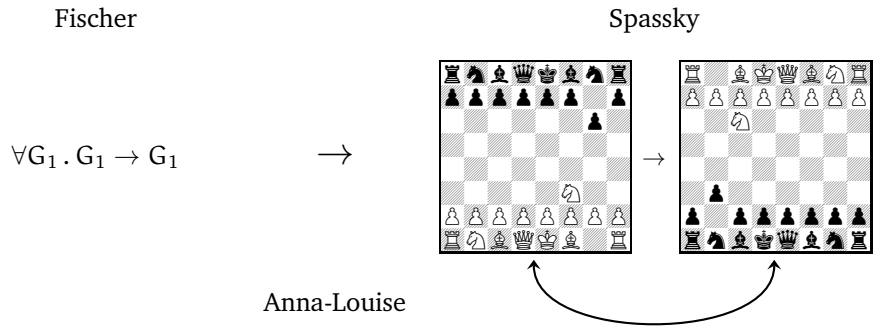
\begin{center}\vspace{2ex}
\hspace*{4ex}\begin{psmatrix}[colsep=4ex,rowsep=0ex]
Fischer && Spassky  \\[6ex]
\rput(0,.1){\rnode{FischerBoard}{$\forall G_1\,.\,G_1\to G_1$}}
&\Large\hspace*{7ex}$\to$\hspace*{7ex}&
\hspace*{4ex}\chesssimulscale{}{1.Nf3}{.45}{8.5ex}
\\[7ex]
&\makebox[0ex]{Anna-Louise}
\\[2ex]
\;
\end{psmatrix}
\end{center}
Spassky has just imported \chess{} and opened with the White move
\wmove{Nf3}.  In this scenario Anna-Louise copies that move locally,
to the other board in front of Spassky:
\begin{center}\vspace{2ex}
\hspace*{4ex}\begin{psmatrix}[colsep=4ex,rowsep=0ex]
Fischer && Spassky  \\[6ex]
\rput(0,.1){\rnode{FischerBoard}{$\forall G_1\,.\,G_1\to G_1$}}
&\Large\hspace*{7ex}$\to$\hspace*{7ex}&
\hspace*{4ex}\chesssimulscale{1.Nf3}{1.Nf3}{.45}{8.5ex}
\\[7ex]
&\makebox[0ex]{Anna-Louise}
\\[2ex]
\;
\end{psmatrix}
\end{center}
Spassky may respond with \wmove{g6} as Black, which
Anna-Louise copies back to the other board:
\begin{center}\vspace{2ex}
\hspace*{4ex}\begin{psmatrix}[colsep=4ex,rowsep=0ex]
Fischer && Spassky  \\[6ex]
\rput(0,.1){\rnode{FischerBoard}{$\forall G_1\,.\,G_1\to G_1$}}
&\Large\hspace*{7ex}$\to$\hspace*{7ex}&
\hspace*{4ex}\chesssimulscale{1.Nf3 g6}{1.Nf3 g6}{.45}{8.5ex}
\\[7ex]
&\makebox[0ex]{Anna-Louise}
\\[2ex]
\;
\end{psmatrix}
\end{center}
She continues to copy moves along the following copycat link, leaving
Fischer to forever twiddle his thumbs:
\begin{center}\vspace{2ex}
\hspace*{4ex}\begin{psmatrix}[colsep=4ex,rowsep=0ex]
Fischer && Spassky  \\[6ex]
\rput(0,.1){\rnode{FischerBoard}{$\forall G_1\,.\,G_1\to G_1$}}
&\Large\hspace*{7ex}$\to$\hspace*{7ex}&
\hspace*{4ex}\chesssimulscale{1.Nf3 g6}{1.Nf3 g6}{.45}{8.5ex}
\\[9ex]
&\makebox[0ex]{Anna-Louise}
\\[3ex]
\;
\end{psmatrix}%
\nccurve[ncurv=.7,arrows=<->,nodesep=.5ex,angleA=-90,angleB=-90]{a}{b}%
\end{center}
This copycat strategy corresponds to the system $F$ term
$$\lambda h^H.\Lambda G.\lambda g^G.g$$
of type $H\to H$.  (Recall $H\,=\,\allGG\,$.)
Fischer's eternal thumb twiddling corresponds to $h$ not showing up in the body of the term.

\paragraph*{The third copycat strategy $\tau\,$.}
The third copycat strategy $\tau$, like the first, the identity
$\iota$, responds to Fischer.  However, instead of importing $\chess$
for $G_1$ against Fischer, as in $\iota$, Anna-Louise imports a
simultaneous chess display $\chess\to\chess$ for $G_1$:\footnote{As
usual, the large arrow {\Large$\to$} between Fischer and Spassky binds
most strongly (so we can omit brackets around the left four boards).}
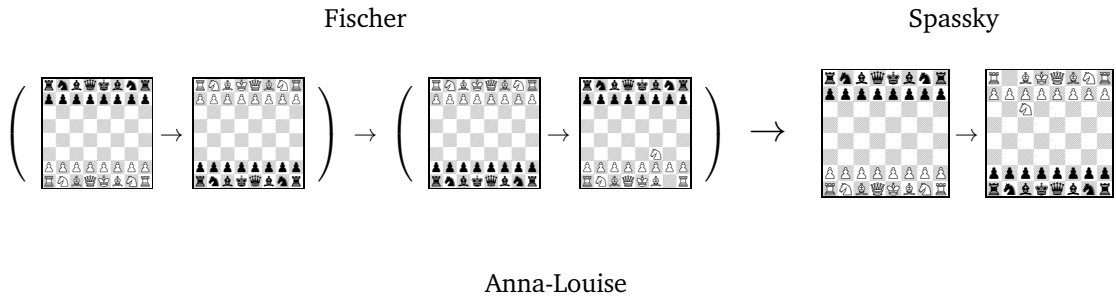
\begin{center}\vspace{2ex}
\begin{psmatrix}[colsep=4ex,rowsep=0ex]
Fischer\hspace*{6ex}&&Spassky\\[4ex]
\hspace*{-6ex}$\left(\hspace*{5.5ex}\mbox{\chesssimulscale{}{}{.26}{5ex}}\hspace*{5.5ex}\right)$%
\hspace*{1ex}{\small$\to$}\hspace*{1ex}%
$\left(\hspace*{5.5ex}\mbox{\invchesssimulscale{}{1.Nf3}{.26}{5ex}}\hspace*{5.5ex}\right)$%
&\Large\;\;$\to$\;\;&
\hspace*{6ex}\chesssimulscale{}{1.Nf3}{.3}{5.5ex}
\\[7ex]
\hspace{24ex}Anna-Louise
\\[2ex]
\;
\end{psmatrix}%
\end{center}
As shown above, Anna-Louise copies Spassky's \wmove{Nf3} onto the
fourth board against Fischer.
She continues with the following geometry of copycat links:
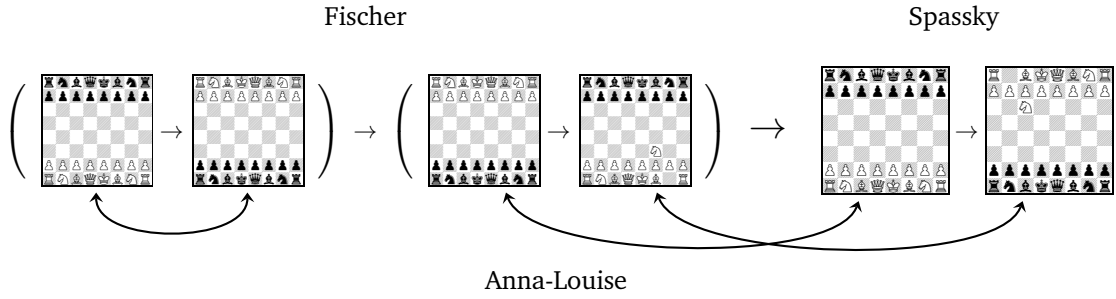
\begin{center}\vspace{2ex}
\begin{psmatrix}[colsep=4ex,rowsep=0ex]
Fischer\hspace*{6ex}&&Spassky\\[4ex]
\hspace*{-6ex}$\left(\hspace*{5.5ex}\mbox{\chesssimulscale{}{}{.26}{5ex}}\hspace*{5.5ex}\right)$%
\nccurve[ncurv=.7,arrows=<->,nodesep=.5ex,angleA=-90,angleB=-90]{a}{b}%
\hspace*{1ex}{\small$\to$}\hspace*{1ex}%
$\left(\hspace*{5.5ex}\mbox{\invchesssimulscale{}{1.Nf3}{.26}{5ex}}\hspace*{5.5ex}\right)$%
&\Large\;\;$\to$\;\;&
\hspace*{6ex}\chesssimulscale{}{1.Nf3}{.3}{5.5ex}
\\[7ex]
\hspace{24ex}Anna-Louise
\\[2ex]
\;
\end{psmatrix}%
\nccurve[ncurv=.4,arrows=<->,nodesep=.5ex,angleA=-70,angleB=-110]{c}{a}
\nccurve[ncurv=.4,arrows=<->,nodesep=.5ex,angleA=-70,angleB=-110]{d}{b}
\end{center}
On the right four boards she continues just as on the four boards of
the identity $\iota$.  If Fischer responds as Black on the fourth
board, she copies this to the last board against Spassky, and if
Fischer opens as White on the third board, she copies this to open the
other board against Spassky.

On the left two boards she ``plays Fischer against himself''.  If
Fischer opens with White on the second board, she copies this to him
on the first board; if Fischer responds as Black on the first board,
she copies that back to the second board.
This corresponds to the first argument $\lambda x^G.x$ of $h_{G\to G}$
in the term
$$\lambda h^H.\Lambda G.\lambda g^G.h_{G\to G}\,(\lambda x^G.x)\,g$$
associated with this strategy.  

Note that all three of the above copycat strategies are uniform: had
the imported game been \ox{} instead of \chess, Anna-Louise would have
copied the moves around in exactly the same geometry.
In the third strategy she would have imported $\ox\to\ox$ for $G_1$
against Fischer.
This strategy always imports $K\to K$ against Fischer, whatever the
game $K$ imported by Spassky.  The geometry of Anna-Louise's six copycat
links is independent of $K$.

\subsection{Other conceptual ingredients of the model}

\emph{This subsection may be somewhat abstruse for readers not already familiar with game semantics; 
consider skipping to Section~\ref{sec-trans} below, without loss of
continuity.}

So far in this introduction we have sketched the following ingredients
of our model:
\begin{itemize}
\item Hypergames: games as moves, to model universal quantification/instantiation.
\item Self-reference: hypergames can be imported into hypergames, and a 
hypergame may even be imported into itself.
\item Uniformity: the shape of Anna-Louise's play, in terms of how 
we copy moves around, cannot depend on the choices of games imported
by the opposing team: she must treat those games as ``black boxes''.  Once
two (sub)games are linked by copycat, she cannot change that link.
\end{itemize}
The following additional ingredients come from prior (first-order,
unquantified) work:
\begin{itemize}
\item Backtracking.  We permit moves to be taken back during play, corresponding to the fact that a system $F$ 
function can call its argument an arbitrary number of times.
Backtracking was used by Lorenzen \cite{Lor60,Fel85} for modelling
proofs of intuitionistic logic, by Coquand \cite{Coq91,Coq95}, and by
Hyland and Ong \cite{HO}.
\item
Innocence.  Following Coquand \cite{Coq91,Coq95}, Hyland-Ong \cite{HO}
and Nickau \cite{Nic96}, strategies depend only on a restricted
``view'' of the history of play.
\item Interaction.  We use Coquand-style interaction between backtracking strategies
to model normalisation of system $F$ terms, specifically, the
refinement by Hyland and Ong of this interaction in a lambda calculus (cartesian closed)
setting.
\item Liveness.  A strategy must always be able to make a move (coined \emph{liveness} by Conway \cite{Con76}).
\item Copycat condition.  We impose (a restriction of) Lorenzen's condition \cite{Lor60} for dialogues
listed by Felscher as (D10) \cite{Fel85}, which requires that an
atomic formula (or in the present system $F$ context, a type variable)
be ``asserted'' by Anna-Louise only if, within her view, the opposing
team has just asserted it.\footnote{I was unaware of Lorenzen's (D10) at the time I wrote \cite{Hug97,Hug00}.}
\end{itemize}
These additional ingredients relate to quantifiers:
\begin{itemize}
\item Copycat expansion.
Technically, uniformity will be implemented by \emph{copycat
expansion} \cite{Hug06h}, similar to Felscher's \emph{skeleton
expansion} \cite{Fel85,Fel01} (and equivalent to the condition in
\cite{Hug97,Hug00}): whenever a strategy includes a play (accepts a
move sequence) $p$, with a variable $X$ imported by the opponent into
a quantified variable, then for all types $T$, all variants of $p$
obtained by substituting $T$ for $X$ and playing copycat between
appropriate instances of $T$ are also in the strategy.\footnote{I was
unaware of Felscher's skeleton expansion at the time I wrote
\cite{Hug97,Hug00}.}
\item Compactness.  A strategy is determined by a finite ``skeleton'', which expresses only the 
copycat links between components.
\end{itemize}
The main theorem is that the map from system $F$ terms to strategies
(satisfying the above properties) is surjective.  A surjectivity
theorem of this kind for simply typed $\lambda$-calculus is given in
\cite{Plo80}, but since \cite{AJ92} such a result in a logical 
setting has often come to be referred to as \emph{full completeness},
when it includes a semantic notion of composition.

\subsubsection{Modular construction of games}

We shall define system $F$ games modularly.  First we define a
transition system whose states are system $F$ types, and whose
transition labels are hypermoves.  The hypermoves involve
instantiating quantifiers in the states (just as the examples above
involved instantiating quantifiers during play).

Every transition system determines a forest (disjoint union of
trees): its set of non-empty traces.
Every forest can be interpreted as an \emph{arena}, in the sense of
Hyland and Ong \cite{HO}.

Following Hyland and Ong, every arena defines a game, with
backtracking.  The (hyper)game we associate with a system $F$ type
will be such a backtracking arena-game.  Since we use arena games,
interaction of strategies (composition) is precisely the Hyland-Ong
interaction.

The underlying first-order composition allows us to relate the
composition to an underlying untyped lambda calculus machine, as in
\cite{DHR96}, upon erasing the system $F$ type information.  In other
words, the composition, when viewed as acting on $\eta$-long
$\beta$-normal forms (representing innocent view functions),
corresponds to (a) erasing the system $F$ type information, (b)
computing with the abstract machine
\cite{DHR96} on the underlying untyped lambda term, then (c) replacing
type information.\footnote{I have a vague recollection that just
such an abstract machine was analysed for system $F$ in the masters'
thesis of Eike Ritter.  I need to investigate this.}
If we erase the type information but stay in the model (\ie, we don't
look at the lambda terms), then we are just composing strategies in a
naive games model of untyped lambda calculus.  The underlying
transition system of the untyped lambda game has a single state and
every integer $i\ge 1$ as transition labels.  These integer labels are
precisely the result of deleting the instatiating types from the
transition labels of the system $F$ transition graph.  Or to put it
another way: the system $F$ transition labels are those of the untyped lambda
transition graph together with type instantiations.
The untyped lambda calculus games similar to those in \cite{KNO02}.

\subsection{Related work}

Affine linear polymorphism was modelled in\cite{Abr97}\footnote{Samson
Abramsky's course at this summer school, during the summer before my
D.Phil., is in part what inspired my choice of thesis topic.} with a
PER-like ``intersections'' of first-order games of the form \cite{AJ94,AJM00}.
Abramsky and Lenisa have explored systematic ways of modelling
quantifiers so that, in the limited case in which all quantifiers are
outermost (so in particular positive), models are fully complete
\cite{AL00}.  (See subsection~\ref{sec-neg} for a simple example of a type 
at which full completeness fails.)

The hypergame/uniformity technique presented here has been applied to
affline linear logic \cite{MO01,Mur01}, and has been used to study
Curry-style type isomorphisms \cite{deL06}.

\section{Transition system games and backtracking}\label{sec-trans}

A game such as \chess{} or \ox{} has a \emph{state} (the configuration
of the board or grid) and, for every state, a set of
\emph{transitions} or moves (\eg\ \wmove{Nf3},
\wmove{Bb4}, $\oxX$ top-right, $\oxO$ centre), each with an ensuing
state.\footnote{For a game of chance such as backgammon, one would
specify a probability distribution over ensuing states, rather than a
single ensuing state.  We consider only deterministic games here.}
Such a game can be specified as a deterministic labelled \emph{transition system}:
an edge-labelled directed graph whose vertices are the states of the
game, with a distinguished initial state.
A fragment of the transition system for chess is illustrated
below.\footnote{The states include data for en passant and castling
and rights, and the turn (Black or White to move), not shown in the
diagram.}
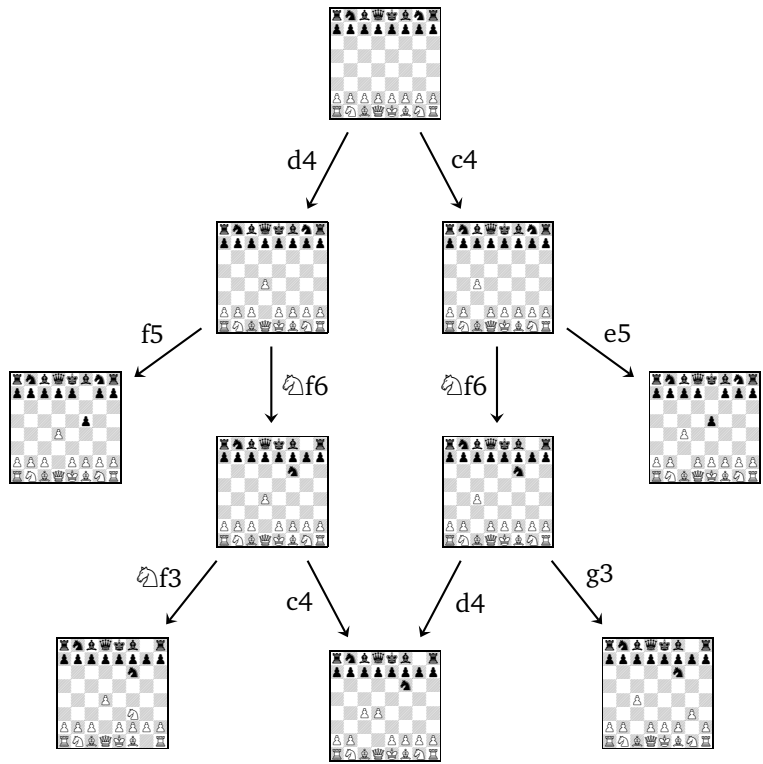
\begin{center}
\pspicture*(-14,-10)(15,.3)
\psset{nodesep=5pt,xunit=.25cm,yunit=.57cm}
\rput(0,-1){\rnode{chess}{\chessposn{}}}
\rput(-6,-6){\rnode{d4}{\chessposn{1.d4}}}
\ncline[arrows=->]{chess}{d4}
\bput(0.5){\wmove{d4}}
\rput(6,-6){\rnode{c4}{\chessposn{1.c4}}}
\ncline[arrows=->]{chess}{c4}
\aput(0.5){c4}
\rput(-17,-9.5){\rnode{d4f5}{\chessposn{1.d4 f5}}}
\ncline[arrows=->]{d4}{d4f5}
\bput(0.5){f5}
\rput(-6,-11){\rnode{d4Nf6}{\chessposn{1.d4 Nf6}}}
\ncline[arrows=->]{d4}{d4Nf6}
\aput(0.5){\wmove{Nf6}}
\rput(17,-9.5){\rnode{c4e5}{\chessposn{1.c4 e5}}}
\ncline[arrows=->]{c4}{c4e5}
\aput(0.5){e5}
\rput(6,-11){\rnode{c4Nf6}{\chessposn{1.c4 Nf6}}}
\ncline[arrows=->]{c4}{c4Nf6}
\bput(0.5){\wmove{Nf6}}
\rput(-14.5,-15.7){\rnode{d4Nf6Nf3}{\chessposn{1.d4 Nf6 2.Nf3}}}
\ncline[arrows=->]{d4Nf6}{d4Nf6Nf3}
\bput(0.5){\wmove{Nf3}}
\rput(14.5,-15.7){\rnode{c4Nf6g3}{\chessposn{1.c4 Nf6 2.g3}}}
\ncline[arrows=->]{c4Nf6}{c4Nf6g3}
\aput(0.5){\wmove{g3}}
\rput(0,-16){\rnode{join}{\chessposn{1.d4 Nf6 2.c4}}}
\ncline[arrows=->]{d4Nf6}{join}
\bput(0.4){c4}
\ncline[arrows=->]{c4Nf6}{join}
\aput(0.4){d4}
\endpspicture
\end{center}
Note that the graph is not a tree.
Without a distinguished initial state, we shall refer to such a graph as a \emph{transition graph}.

Formally, a \defn{transition graph} $(Q,M,\trans)$ comprises a set $Q$
of \defn{states}, a set $L$ of \defn{labels}, and a partial \defn{transition function}
$\trans:Q\times L\pto Q$.\footnote{We
write $f:X\pto Y$ if $f$ is a partial function from $X$ to $Y$, \ie,
a function $X'\to Y$ for some $X'\subseteq X$.}
We write $q\transto{l}q'$ for $\trans(q,l)=q'$.
A \defn{transition system} $(Q,L,\trans,\star)$ is a transition graph
$(Q,L,\trans)$ together with an \defn{initial state} $\star\in Q$.
A \defn{trace} of $(Q,L,\trans,\star)$
is a finite sequence $l_1\ldots l_k$ of labels $l_i\in L$
$(k\ge 0)$ such that
$$\star\transto{l_1}q_1\transto{l_1}q_2\transto{l_3}\cdots\transto{l_{k-1}}
q_{k-1}\transto{l_k}q_k$$ for states $q_i\in Q$ ($1\le i\le
k$).\footnote{Note that the states $q_i\in Q$ are uniquely determined
by the $l_i$, since our transition systems are implicitly
deterministic.}
For example, \hspace{1ex}\wmove{d4}\;\wmove{Nf6}\;\wmove{c4}\hspace{1ex} is a trace of
chess, visible in the diagram above.

\subsection{Games}

Let $M$ be a set of \defn{moves}.  A \defn{trace over $M$} or
\defn{$M$-trace} is a list (finite sequence) $m_1\ldots m_k$ of moves $m_i\in
M$ ($k\ge 0$).
A set $G$ of $M$-traces is a \defn{tree} if whenever $m_1\ldots m_k$
is in $G$ with $k\ge 1$ then its \defn{predecessor} $m_1\ldots
m_{k-1}$ is also in $G$, and the empty trace $\emptyseq$ is in $G$ (the
root of the tree).
A \defn{game over $M$} or \defn{$M$-game} is a tree of $M$-traces.
Following \cite{HO}, we write $\opp$ for the first player
(associated with odd moves,
\ie, moves in a trace with odd index), and $\pla$ for the second player
(associated with even moves).

Every transition system $\Delta$ with label set $M$ defines a game
$G(\Delta)$ over $M$, namely the set of traces of $\Delta$.
For example, if $\chesssys$ is the chess transition system
depicted above, and $\chessmov$ is the set of all chess moves
$\{\wmove{Ne5},\wmove{Qa2},\wmove{Kh1},\ldots\}$, then
$G(\chesssys)$ (the set of all traces of the chess transition
system) is a game over $\chessmov$.
This game comprises all legal sequences of chess moves.

\subsection{Strategies}

A \defn{strategy} (implicitly for the second player $\pla$) for a game
$G$ is a tree $\sigma\subseteq G$ whose every odd-length trace has a
unique one-move extension in $\sigma$: if $m_1\ldots m_k\in \sigma$
and $k$ is odd,
there exists a unique move $m$ such that $m_1\ldots m_k m\in
\sigma$.
A strategy $\sigma$ for $G$ is \defn{live} (or \defn{total}) if it responds to every
stimulus: if $m_1\ldots m_k\in \sigma$ with $k$ even and $m_1\ldots
m_km\in G$, then $m_1\ldots m_km\in
\sigma$.\footnote{Thus $m_1\ldots m_kmn\in \sigma$ for a unique $n$,
the ``response of $\sigma$ to $m$ after $m_1\ldots m_k$''.  One is
also tempted to call such a strategy \emph{total}, by analogy with
partial versus total functions; we shall stick with Conway's original
terminology \cite{Con76}.}

\subsection{Backtracking}

When playing chess against a computer, there is usually an option
to take a move back.  If we allow both players (user and computer) to
take back moves, and also to return to previously abandoned lines, we
obtain a derived game in which a move is either an opening chess move
(starting or restarting the game) or is a pair: a pointer to an
earlier move by the opponent, and a chess move in response to that
move.  For example, here is a trace of backtracking chess, with time
running left-to-right (so backtracking pointers are right-to-left):
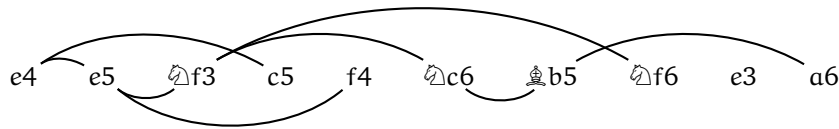
\begin{center}\vspace{4ex}\begin{math}
\psset{nodesep=2pt,colsep=4ex,nodesep=1.5pt}\begin{psmatrix}
\rnode{1}{\wmove{e4}}
&
\rnode{2}{\wmove{e5}}
&
\rnode{3}{\wmove{Nf3}}
&
\rnode{4}{\wmove{c5}}
&
\rnode{5}{\wmove{f4}}
&
\rnode{6}{\wmove{Nc6}}
&
\rnode{7}{\wmove{Bb5}}
&
\rnode{8}{\wmove{Nf6}}
&
\rnode{9}{\wmove{e3}}
&
\rnode{10}{\wmove{a6}}
\end{psmatrix}
\jptr{2}{1}{150}{35}
\jptr{3}{2}{-140}{-40}
\jptr{4}{1}{150}{35}
\jptr{5}{2}{-140}{-40}
\jptr{6}{3}{150}{30}
\jptr{7}{6}{-140}{-40}
\jptr{8}{3}{150}{30}
\jptr{10}{7}{150}{30}
\end{math}\vspace{4ex}\end{center}
The penultimate move \wmove{e3}, with no backtracking pointer, is a
restarting move.
Since this is a trace of a game with an underlying transition system,
we can include the states in the depiction, as below, which
corresponds to the first six moves above.
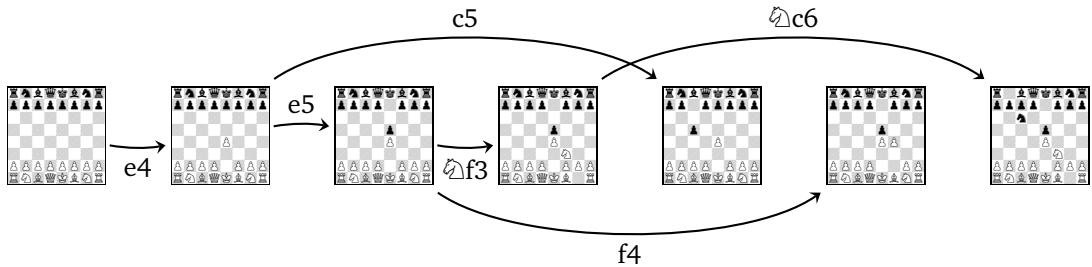
\begin{center}\vspace{6ex}\newcommand{\scal}{0.23}\label{chessdiagstates}
\psset{nodesep=2pt,colsep=5ex,nodesep=1.5pt}\begin{psmatrix}
\rnode{0}{\chessposnscale{\scal}{}}
&
\rnode{1}{\chessposnscale{\scal}{1.e4}}
&
\rnode{2}{\chessposnscale{\scal}{1.e4 e5}}
&
\rnode{3}{\chessposnscale{\scal}{1.e4 e5 2.Nf3}}
&
\rnode{4}{\chessposnscale{\scal}{1.e4 c5}}
&
\rnode{5}{\chessposnscale{\scal}{1.e4 e5 2.f4}}
&
\rnode{6}{\chessposnscale{\scal}{1.e4 e5 2.Nf3 Nc6}}
\end{psmatrix}
\nccurve[arrows=<-,angleA=-170,angleB=-10]{1}{0}\aput(0.5){\wmove{e4}}
\nccurve[arrows=<-,angleA=170,angleB=10]{2}{1}\bput(0.5){\wmove{e5}}
\nccurve[arrows=<-,angleA=-170,angleB=-10]{3}{2}\aput(0.5){\wmove{Nf3}}
\nccurve[ncurv=.5,nodesep=5pt,arrows=<-,angleA=145,angleB=35,offsetA=-1.2ex,offsetB=-1.2ex]{4}{1}\bput(0.5){\wmove{c5}}
\nccurve[ncurv=.5,nodesep=5pt,arrows=<-,angleA=-145,angleB=-35,offsetA=1.2ex,offsetB=1.2ex]{5}{2}\aput(0.5){\wmove{f4}}
\nccurve[ncurv=.5,nodesep=5pt,arrows=<-,angleA=145,angleB=35,offsetA=-1.2ex,offsetB=-1.2ex]{6}{3}\bput(0.5){\wmove{Nc6}}
\vspace{6ex}\end{center}
In this depiction we draw the pointers akin to transitions in the
underlying transition system, with their labels.  This clarifies the
sense in which we refer back to a previous state during backtracking,
and make our move from there.

We shall write $\backtrack{G}$ for the backtracking variant of a game $G$, formalised below.
Let $M$ be a set of moves.  A 
\defn{dialogue} over $M$ is a an $M$-trace in which each
element may carry an odd length pointer to an earlier element (\cf\ \cite{Lor60,Fel85,Coq91,Coq95,HO}).
For example, a dialogue over the set $\chessmov$ of chess
moves is depicted above.
Formally, a dialogue over $M$ is an ($\Nat\!\times\!M$)-trace\footnote{$\Nat=\{0,1,2,\ldots\}$.}
$$\langle \alpha_1,m_1\rangle\ldots\langle \alpha_k,m_k\rangle$$ such
that $i-\alpha_i\in\{1,3,5,\ldots\}$ for $1\le i\le
k$.
Each $\alpha_i$ represents a pointer from $m_i$ back to
$m_{\alpha_i}$, with $\alpha_i=0$ coding ``$m_i$ has no pointer''.
The formalisation of the chess dialogue depicted above is the
following ($\Nat\!\times\!\chessmov$)-trace:
\begin{center}\begin{math}
\langle 0,\wmove{e4}\rangle\;\langle 1,\wmove{e5}\rangle\;
\langle 2,\wmove{Nf3}\rangle\; \langle 1,\wmove{c5}\rangle\; \langle 2,\wmove{f4}\rangle\; \langle 3,\wmove{Nc6}\rangle \; \langle 6,\wmove{Bb5}\rangle
\; \langle 3,\wmove{Nf6}\rangle \; \langle 0,\wmove{e3}\rangle \; \langle 7,\wmove{a6}\rangle
\end{math}\end{center}
A move of the form $\langle 0,m\rangle$, without a pointer, is a
\defn{starting} move.
A \defn{thread} of a dialogue over $M$ is any sequence of elements
traversed from a starting move by following pointers towards the
right.
For example,
$\wmove{e4}\:\wmove{e5}\:\wmove{Nf3}\:\wmove{Nc6}\:\wmove{Bb5}$ is a
thread of the chess dialogue above:
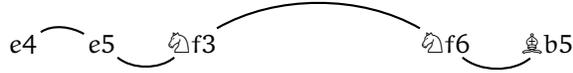
\begin{center}\vspace{2ex}\begin{math}
\psset{colsep=4ex,nodesep=1.5pt}\begin{psmatrix}
\rnode{1}{\wmove{e4}}
&
\rnode{2}{\wmove{e5}}
&
\rnode{3}{\wmove{Nf3}}
&
{\rule{2ex}{0ex}\rule{0ex}{1ex}}
&
{\rule{2ex}{0ex}\rule{0ex}{1ex}}
&
\rnode{6}{\wmove{Nf6}}
&
\rnode{7}{\wmove{Bb5}}
&
{\rule{2ex}{0ex}\rule{0ex}{1ex}}
&
{\rule{2ex}{0ex}\rule{0ex}{1ex}}
&
{\rule{2ex}{0ex}\rule{0ex}{1ex}}
\end{psmatrix}
\jptr{2}{1}{150}{35}
\jptr{3}{2}{-140}{-40}
\jptr{6}{3}{150}{30}
\jptr{7}{6}{-140}{-40}
\end{math}\vspace{3ex}\end{center}
The singleton sequence $\wmove{e3}$ is also a thread, as is \wmove{e4}\,\wmove{e5}\,\wmove{f4}\,.
Formally, an $M$-trace $m_{d_1}\ldots m_{d_n}$ (where $n\ge 0$) is a
\defn{thread} of
the dialogue $\langle \alpha_1,m_1\rangle\cdots\langle
\alpha_k,m_k\rangle$ over $M$ if $\alpha_{d_1}=0$ and
$\alpha_{d_j}=d_{j-1}$ for $1<j\le n$.

Let $G$ be an $M$-game.  A dialogue over $M$ \defn{respects $G$} if
its threads are in $G$.
For example, if $\chess$ abbreviates our earlier formalisation
$G(\chesssys)$ of the game of chess as a transition system game,
then the dialogue over $\chessmov$ depicted above respects
$\chess$ (since every thread is a legal sequence of chess moves from the
initial chess position).
The \defn{backtracking game} $\backtrack G$ is
the set of all dialogues over $M$ which respect $G$.
For example, the dialogue over $M_{\chess}$ depicted above is a trace
of $\backtrack{\chess}$, \ie, of ``backtracking chess''.

The \defn{$\pla$-backtracking game} $\pbacktrack{G}$ is obtained from
$\backtrack G$ by permitting only the second player $\pla$ to
backtrack: every $\opp$-move (move in odd position) but the first
points to the previous move.  Formally, $\pbacktrack{G}$ comprises
every $\langle \alpha_1,m_1\rangle\ldots\langle \alpha_k,m_k\rangle$
in $\backtrack G$ such that $\alpha_i=i-1$ for all odd
$i\in\{1,\ldots,k\}$.
A dialogue of $\pbacktrack{\chess}$ is shown below.
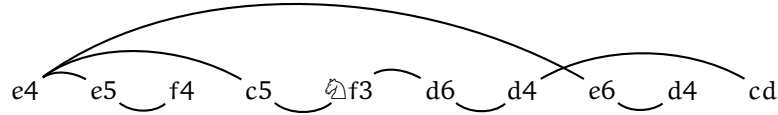
\begin{center}\vspace{6ex}\begin{math}
\psset{nodesep=2pt,colsep=4ex,nodesep=1.5pt}\begin{psmatrix}
\rnode{1}{\wmove{e4}}
&
\rnode{2}{\wmove{e5}}
&
\rnode{3}{\wmove{f4}}
&
\rnode{4}{\wmove{c5}}
&
\rnode{5}{\wmove{Nf3}}
&
\rnode{6}{\wmove{d6}}
&
\rnode{7}{\wmove{d4}}
&
\rnode{8}{\wmove{e6}}
&
\rnode{9}{\wmove{d4}}
&
\rnode{10}{\wmove{cd}}
\end{psmatrix}
\jptr{2}{1}{150}{35}
\jptr{3}{2}{-140}{-40}
\jptr{4}{1}{150}{36}
\jptr{5}{4}{-140}{-40}
\jptr{6}{5}{150}{30}
\jptr{7}{6}{-140}{-40}
\jptr{8}{1}{150}{36}
\jptr{9}{8}{-145}{-40}
\jptr{10}{7}{150}{30}
\end{math}\vspace{2ex}\end{center}

For every type $T$ of system $F$, we shall define a transition system
$\Delta_T$ and define the hypergame associated with $T$ simply as the
backtracking game over this transition system, \ie,
$\backtrack{G(\Delta_T)}$.
For didactic purposes, we begin in the next section with the
restricted case of lambda calculus.

\section{Lambda calculus games}

Let $\lambdax$ denote the types of $\lambda$ calculus generated from
a single base type $X$ by implication $\to$.  
Every $\lambdax$ type $T$ determines a transition system $\Delta_T$:
\begin{itemize}
\item 
States are $\lambdax$ types, with an additional initial state $\dummy\,$.
\item A label is any $i\in\{1,2,3,\ldots\}$, called a \defn{branch choice}.
\item Transitions.  A $1$-labelled transition
$$\dummy\hspace{1ex}\transto{1}\hspace{1ex}T$$
from the initial state to $T$, and transitions
$$T_1\to T_2\to\ldots\to T_n\to X\hspace{3ex}\transto{i}\hspace{3ex}T_i{_{_{\rule{0ex}{2ex}}}}\hspace*{8ex}$$
for $1\le j\le n$.
\end{itemize}
For example, if 
$U\;=\;X\to(X\to X)\to X$
then the reachable portion of the transition system $\Delta_U$ is
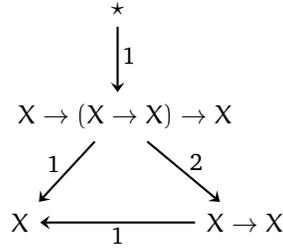
\begin{center}\vspace{1ex}\begin{math}
\psset{colsep=-2ex,rowsep=6ex,nodesep=.9ex,labelsep=1pt}\begin{psmatrix}
{} & \rnode{q0}{\dummy} & {} \\
{} & \rnode{q1}{\;\;X\to(X\to X)\to X}     & {} \\
\rnode{q2}{X} & {} & \rnode{q3}{X\to X}
\transbetweena{q0}{q1}{1}{.45}
\transbetweenb{q1}{q2}{1}{.55}
\transbetweena{q1}{q3}{2}{.55}
\transbetweena{q3}{q2}{1}{.5}
\end{psmatrix}\end{math}\vspace{2ex}\end{center}
so the associated (non-backtracking) game (set of traces) $G(\Delta_U)$ is
$\{ \emptyseq, 1, 11, 12, 121 \}$,
where $\emptyseq$ denotes the empty sequence.
\begin{theorem}\label{theoremX}
Let $T$ be a lambda calculus type generated from a single base type
$X$ by implication $\to$.  The $\eta$-expanded $\beta$-normal terms of type $T$
are in bijection with finite live
strategies on the $\player$-backtracking game $\pbacktrack
{G(\Delta_T)}$.
\end{theorem}
\begin{proof} A routine induction: a restriction of the definability proof in \cite{Hug97}, in turn a variant of the definability proof in \cite{HO}.
\end{proof}
The $\eta$-expanded $\beta$-normal forms $t_n$ of $U\,=\,X\to (X\to
X)\to X$ (whose transition system was depicted above)
are\footnote{$f^n$ denotes $n$ applications of $f$: $\:f^0(x)=x$ and
$f^{n}(x)=f(f^{n-1}(x))$ for $n\ge 1$.}
\begin{displaymath}
\lambda x^X.\,\lambda f^{X\to X}.\,f^n(x)
\end{displaymath}
for $n\ge 0$
and the unique maximal trace of the corresponding live finite
strategy $\tau_n$ 
on $\pbacktrack{G(\Delta_T)}$ is
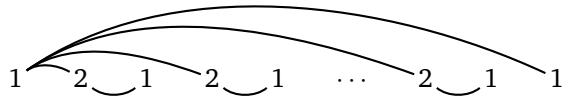
\begin{center}\vspace{5ex}\begin{math}
\psset{colsep=4ex,nodesep=1.5pt}\begin{psmatrix}
\rnode{1}{1}
&
\rnode{2}{2}
&
\rnode{3}{1}
&
\rnode{4}{2}
&
\rnode{5}{1}
&
\cdots
&
\rnode{8}{2}
&
\rnode{9}{1}
&
\rnode{10}{1}
\end{psmatrix}
\jptr{2}{1}{150}{35}
\jptr{3}{2}{-140}{-40}
\jptr{4}{1}{160}{35}
\jptr{5}{4}{-140}{-40}
\jptr{8}{1}{160}{35}
\jptr{9}{8}{-140}{-40}
\jptr{10}{1}{155}{35}
\end{math}\vspace{1ex}\end{center}
with $n$ occurrences of $2$.
Below we depict this dialogue in the case $n=2$ (corresponding to the
term $\lambda x^X\,.\,\lambda f^{X\to X}\,.\,f(f\,x)$) with its states (as we
did for the chess dialogue on page~\pageref{chessdiagstates}).  It is
easier to display with time running down the page, rather than from
right to left.
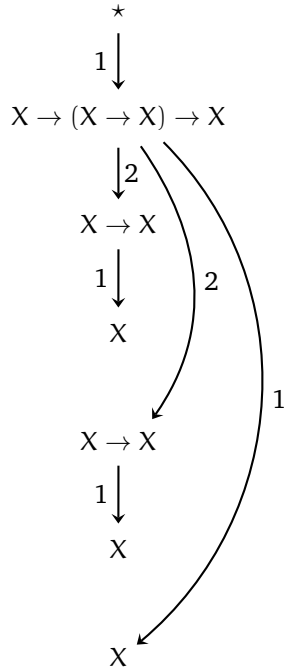
\begin{center}\vspace{1ex}
\psset{nodesep=6pt,colsep=5ex,rowsep=6ex}\begin{psmatrix}
\rnode{0}{$\star$}
\\
\rnode{1}{$X\to(X\to X)\to X$}
\\
\rnode{2}{$X\to X$}
\\
\rnode{3}{$X$}
\\
\rnode{4}{$X\to X$}
\\
\rnode{5}{$X$}
\\
\rnode{6}{$X$}
\\
\end{psmatrix}
\nccurve[arrows=<-,angleA=90,angleB=-90]{1}{0}\aput(0.5){1}
\nccurve[arrows=<-,angleA=90,angleB=-90]{2}{1}\bput(0.5){\!2}
\nccurve[arrows=<-,angleA=90,angleB=-90]{3}{2}\aput(0.5){1}
\nccurve[ncurv=.8,nodesepB=6pt,nodesepA=8pt,arrows=<-,angleA=55,angleB=-55,offsetA=-1.2ex,offsetB=-.2ex]{4}{1}\bput(0.5){2}
\nccurve[arrows=<-,angleA=90,angleB=-90]{5}{4}\aput(0.5){1}
\nccurve[ncurv=.8,nodesepB=8pt,,nodesepA=3pt,arrows=<-,angleA=40,angleB=-45,offsetA=-.2ex,offsetB=-1.2ex]{6}{1}\bput(0.5){1}
\vspace{4ex}\end{center}
This notation highlights the similarity with Lorenzen's
dialogues \cite{Lor60,Fel85}.  What we show as states, he referred to as
\emph{assertions}.  

\subsection{The copycat condition}

In this section we introduce the \defn{copycat condition} \cite{Hug97}
on strategies, which is crucial for uniformity (more precisely, for us
to be able to implement uniformity via \emph{copycat expansion}
later).  This condition a slight restriction of a condition of
Lorenzen for dialogue games (listed as condition (D10) in
\cite{Fel85}).  We shall introduce the condition in the context of
lambda calculus games; the generalisation to system $F$ games in the
sequel is immediate.

Extend the set $\lambdax$ of lambda calculus types from the previous
subsection to those generated by implication $\to$ from the ambient set
$\mathsf{Var}$ of system $F$ type variables.  The transition system
$\Delta_T$ associated with a $\lambdax$ type $T$ is defined exactly as
in the previous section, but now in the transitions 
$$T_1\to T_2\to\ldots\to T_n\to
X\hspace{3ex}\transto{i}\hspace{3ex}T_i{_{_{\rule{0ex}{2ex}}}}\hspace*{8ex}$$
$X$ may be \emph{any} type variable in $\mathsf{Var}$.

The \defn{colour} of a transition
$$T_1\to\ldots\to T_n\to
X\hspace{3ex}\transto{i}\hspace{3ex}U_1\to\ldots\to U_m\to Y$$
(where necessarily $T_i\,=\,U_1\to\ldots\to U_m\to Y$) is the rightmost variable
$Y$ in the target.
The colour of a move in a trace of $G(\Delta_T)$ or a dialogue in
$\backtrack{G(\Delta_T)}$ is the colour of the associated transition.
A dialogue in the $\pla$-backtracking game $\pbacktrack{G(\Delta_T)}$
satisfies the \defn{copycat condition} if the colour of every
$\pla$-move (even-index move) is equal to the colour of the preceding
$\opp$-move.\footnote{Lorenzen's condition (D10) required the colour
to be equal to \emph{any} prior $\opp$-move in a $\pla$-backtracking
trace.}  A strategy satisfies the copycat condition if each of its
traces satisfies the copycat condition.

As a simple illustration of the copycat condition, consider the type
$$U\;\;\;=\;\;\; X\to Y\to X$$ whose transition system $\Delta_U$ is
below (only reachable states shown).
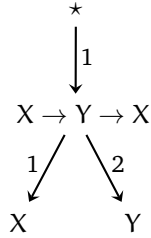
\begin{center}\vspace{1ex}\begin{math}
\psset{colsep=-2ex,rowsep=6ex,nodesep=.9ex,labelsep=1pt}\begin{psmatrix}
{} & \rnode{q0}{\dummy} & {} \\
{} & \rnode{q1}{\;\;X\to Y\to X}     & {} \\
\rnode{q2}{X} & {} & \rnode{q3}{Y}
\transbetweena{q0}{q1}{1}{.45}
\transbetweenb{q1}{q2}{1}{.55}
\transbetweena{q1}{q3}{2}{.55}
\end{psmatrix}\end{math}\vspace{2ex}\end{center}
The colour of the top and lower-left transitions is $X$, and the
colour of the lower-right transition is $Y$.
The associated (non-backtracking) game (set of traces) $G(\Delta_U)$ is
$\{ \emptyseq, 1, 2, 11, 12 \}$.
There are two live strategies in the $\pla$-backtracking game
$\pbacktrack{G(\Delta_U)}$, whose maximal traces are as follows, with
the colour of each move shown beneath it in brackets:
\begin{center}\vspace{2ex}\begin{math}
\psset{colsep=4ex,nodesep=1.5pt,rowsep=0ex}\begin{psmatrix}
\rnode{1}{1}
&
\rnode{2}{1}
\\
(X) & (X)
\end{psmatrix}
\jptr{2}{1}{150}{35}
\hspace*{20ex}\begin{psmatrix}
\rnode{1}{1}
&
\rnode{2}{2}
\\
(X) & (Y)
\end{psmatrix}
\jptr{2}{1}{150}{35}
\end{math}\vspace{1ex}\end{center}
The first strategy satisfies the copycat condition, while the second
does not.  The strategies correspond (respectively) to the terms
$$\lambda x^X.\lambda y^Y.x\hspace{23ex}\lambda x^X.\lambda y^Y.y$$
of which only the former is typed correctly as $X\to Y\to X$.  The
second attempts to return $y$ of type $Y$, while the rightmost
variable of $X\to Y\to X$ is $X$.  This corresponds to the failure of
the copycat condition for the second strategy.

The following is a corollary of the theorem above.
\begin{theorem}\label{theoremXY}
Let $T$ be a lambda calculus type generated from the set
$\textsf{Var}$ of system $F$ type variables by implication $\to$.  The
$\eta$-expanded $\beta$-normal terms of type $T$ are in bijection with
finite live strategies on the $\player$-backtracking game $\pbacktrack
{G(\Delta_T)}$ which satisfy the copycat condition.
\end{theorem}

\subsection{Remarks on Hyland-Ong arenas}

\emph{This section is for readers familiar with Hyland-Ong
games \cite{HO}.  It can be skipped without loss of continuity.}

The set of non-empty traces of a transition system forms a forest
under the prefix order, and is therefore an \defn{arena} in the sense
of Hyland and Ong \cite{HO}.
Write $\arena(\Delta)$ for the arena of a transition system $\Delta$,
and for a lambda calculus type $T$ abbreviate $\arena(\Delta(T))$ to
$\arena(T)$.  
The following arena isomorphism is immediate:
\begin{displaymath}
\begin{array}{r@{\;\;\;\;\;\cong\;\;\;\;\;}l}
\arena(T\to U) & \mathsf{A}(T)\Rightarrow \arena(U) \\[1ex]
\end{array}
\end{displaymath}
where $\Rightarrow$ is the Hyland-Ong function space operation on
arenas and $\cong$ is isomorphism of forests.

Elements of these arenas are sequences (traces), and therefore Hyland-Ong dialogues in them suffer
some redundancy, as in (for example) 
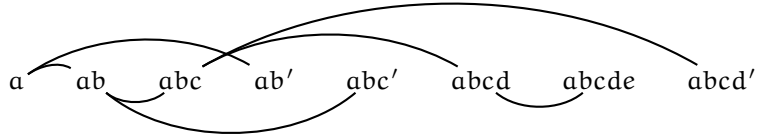
\begin{center}\vspace{4ex}\begin{math}
\psset{nodesep=2pt,colsep=4ex,nodesep=1.5pt}\begin{psmatrix}
\rnode{1}{a}
&
\rnode{2}{ab}
&
\rnode{3}{abc}
&
\rnode{4}{ab'}
&
\rnode{5}{abc'}
&
\rnode{6}{abcd}
&
\rnode{7}{abcde}
&
\rnode{8}{abcd'}
\end{psmatrix}
\jptr{2}{1}{150}{35}
\jptr{3}{2}{-140}{-40}
\jptr{4}{1}{150}{35}
\jptr{5}{2}{-140}{-40}
\jptr{6}{3}{150}{30}
\jptr{7}{6}{-140}{-40}
\jptr{8}{3}{150}{30}
\end{math}\vspace{4ex}\end{center}
in the arena generated by a transition system with transition labels
$a,b,b',c,c',d,d'$, whose traces include $abcde$, $abcd'$, $abc'$,
\etcet\
Clearly, one can abbreviate this trace to
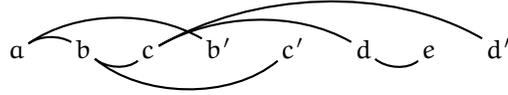
\begin{center}\vspace{4ex}\begin{math}
\psset{nodesep=2pt,colsep=4ex,nodesep=1.5pt}\begin{psmatrix}
\rnode{1}{a}
&
\rnode{2}{b}
&
\rnode{3}{c}
&
\rnode{4}{b'}
&
\rnode{5}{c'}
&
\rnode{6}{d}
&
\rnode{7}{e}
&
\rnode{8}{d'}
\end{psmatrix}
\jptr{2}{1}{150}{35}
\jptr{3}{2}{-140}{-40}
\jptr{4}{1}{150}{35}
\jptr{5}{2}{-140}{-40}
\jptr{6}{3}{150}{30}
\jptr{7}{6}{-140}{-40}
\jptr{8}{3}{150}{30}
\end{math}\vspace{4ex}\end{center}
eliminating the redundancy.  This is how we have opted to formalise
the backtracking games over transition systems in the previous
subsections.  Note, however, this notational difference is trivial,
and in spirit they are essentially Hyland-Ong arena/dialogue games.  The
notation is simply taylored towards arenas whose forests are described
as sets of traces, rather than partial-order forests as used
originally by Hyland and Ong
\cite{HO}.
Since our games are Hyland-Ong games, and we have the isomorphism
relating syntactic $\to$ with arena $\Rightarrow$ above, we obtain
composition (hence a category) as standard Hyland-Ong
composition of innocent strategies.

In the next section we extend the lambda calculus transition systems
to system $F$ transition systems.  The following arena (forest)
ismorphisms will then hold:
\begin{displaymath}
\begin{array}{r@{\;\;\;\;\;\cong\;\;\;\;\;}l}
\arena(T\to U) & \mathsf{A}(T)\Rightarrow \arena(U) \\[1ex]
\arena(\forall X.T) & \prod_{\mathsf{Types }\,U}\arena(T[U/X])
\end{array}
\end{displaymath}
The arena-product $\prod$ (disjoint union of forests) is taken over
all system $F$ types.  Composition in our system $F$ model is simply
Hyland-Ong composition.

\section{System $F$ games (hypergames)}

We extend the lambda calculus transition systems defined above to all
of system $F$.  States will be types, as before, and a transition will
remain a branch choice $i\ge 1$, but now together with some types
to instantiate quantifiers.  We begin by precisely defining
quantifier instantiation.

Recall that a \defn{prenex type} is a type in which all quantifiers have
been pulled to the front by exhaustively applying the
rewrite $$T\to
\forall X.U\;\;\;\leadsto\;\;\; \forall X\,.\,T\to U$$
throughout the type.\footnote{Without loss of generality, in the rewrite assume $X$ is not free in
$T$.}
Thus a type is prenex if and only if it has the form
$$\forall X_1\,.\forall X_2\,.\,\cdots\,\forall X_m\,.\:T_1\to T_2\to \ldots\to T_n\to X$$ 
for prenex types $T_i$ and type variables $X$ and $X_j$.
Write $T[V/X]$ for the result of substituting the type $V$ for
the free variable $X$ throughout the prenex type $T$, and (if
necessary) converting to prenex form.  For example 
$$(X\to X)[\forall Y.Y/X]\;\;\;\;=\;\;\;\;\forall Y.\:(\forall Y'.Y')\to Y$$
via: $$(X\to X)[\forall Y.Y/X]\;\;\;\overset{\text{substitute}}{\leadsto}\;\;\; (\forall Y.Y)\to (\forall Y.Y)
\;\;\;\overset{\text{prenex}}{\leadsto}\;\;\;\forall Y.(\forall Y'.Y')\to Y\,.$$
Define $$\forall X\,.T\:\import\: V\;\;\;\;=\;\;\;\; T[V/X]$$ called
the result of \defn{importing} $V$ into $\forall X.T$.  For example, 
$$\forall X\,.\,X\to X\;\import\; \forall Y.Y\;\;\;\;=\;\;\;\;\forall Y.\:(\forall Y'.Y')\to Y\,.$$
Write $T\import V_1V_2\ldots V_n$ for the iterated importation
$(\ldots((T\import V_1)\import V_2)\import\ldots)\import V_n$, when 
defined.  For example, 
$$\forall X\,.\,X\to X\;\import\;(\forall Y.Y)\;(Z\to Z)\;\;\;=\;\;\;(\forall Y.Y)\to Z\to Z\,.$$

A prenex type is \defn{resolved} if it has no outermost quantifier,
\ie, 
it has the form $$T_1\to T_2\to \ldots\to T_n\to X\,,$$
a form which we shall often abbreviate to
$$T_1T_2\ldots T_n\;\;\to\;\;X$$
Each $T_i$ is called a \defn{branch}.
If $T\import U_1\ldots U_n$ is resolved,
we say that $U_1\ldots U_n$ \defn{resolves} $T$
to $T\import U_1\ldots U_n$.
For example, we saw above that $(\forall Y.Y)(Z\to Z)$ resolves
$\forall X\,.\,X\to X$ to $(\forall Y.Y)\to Z\to Z$.

Define the transition system $\Delta_T$ of a prenex type $T$ as follows:
\begin{itemize}
\item 
States are resolved prenex types, with an additional initial state
$\dummy$\,.\footnote{Prenex types were drawn graphically in
\cite{Hug97,Hug00}, in a manner akin to B\"ohm trees, and called
\emph{polymorphic arenas}.}
\item
A label is a pair $\langle i,V_1\ldots V_k\rangle$ where $i\ge 1$ is a
\defn{branch choice}, $k\ge 0$ and each $V_i$ is a type, called an \defn{import}.
\item Transitions.  A $1$-labelled transition
$$\dummy\hspace{1ex}\transto{1}\hspace{1ex}T$$ from the initial state
to $T$, and transitions
$$\rule{0ex}{3ex}T_1T_2\ldots T_n\to X\;\;\;\;\;\;\;\transto{\langle
i,V_1\ldots V_k\rangle}\;\;\;\;\;\;\;U_1U_2\ldots U_m\to Y$$ whenever $1\le i\le n$
and
$$T_i\import V_1\ldots V_k\;\;\;=\;\;\;U_1U_2\ldots U_m\to Y$$
(Thus a transition chooses a branch $T_i$ and resolves it to
form the next state.)
\end{itemize}
More generally, the transition system of a type is the transition system of its prenex form.

An example transition is shown below.
$$(X'\to X'\to X')\,\to\,(\forall X.X\to X)\,\to\,X''
\hspace*{10ex}
\transto{\langle 2,(\forall Y.Y)(Z_1\!\to\! Z_2\!\to\! Z)\rangle}
\hspace*{10ex}
(\forall Y.Y)\,\to\,Z_1\,\to\, Z_2\,\to Z$$
The branch choice $2$ selects the branch $\forall X.X\to X$ and the
imports $\forall Y.Y$ and $Z_1\!\to\! Z_2\!\to Z$ resolve this branch
to form the next state.\footnote{To obtain a category with products,
we extend system $F$ with products, and allow
import/resolution/substition with products.}

\subsection{Implicit prenexification}\label{lazy}

Prenexification is a lynchpin of the hypergames approach \cite{Hug97}:
it is critical to the dynamics of hypergames that in a type
$T\to \forall X.U$
the quantifier $\forall X$ is available for instantiation.
Whether we make the prenexifications $T\to\forall
X.U\;\leadsto\;\forall X.T\to U$ explicit during play or not is
optional.  
We can just as well leave prenexification implicit, by formally
designating $\forall X$ as available for instantiation in
$T\to\forall X.U\,$.

A quantifier $\forall X$ in a type $T$ is \defn{available} if $T$ has
any of the following forms:\footnote{We assume without loss of
generality throughout this section that all bound variables are
distinct from one another and from the free variables.}
\begin{itemize}
\item $T\,=\,\forall X.U$
\item $T\,=\,U\to T'$ and $\forall X$ is available in $T'$
\item $T\,=\forall Y. T'$ and $\forall X$ is available in $T'$.
\end{itemize}
For example, $\forall X$ and $\forall Y$ are available in 
$\forall Y.(\forall Z.Y\to Z)\to \forall X.X$, 
but $\forall Z$ is
not.\footnote{Note that $\forall X$ is available in $T$ iff it is one
of the outermost quantifiers in the prenex form $\widetilde T$ of $T$
(\ie, $\widetilde T\,=\,\forall X_1\ldots \forall X_k.U$ and $X$ is
among the $X_i$).  In this sense, prenexification is implicit, or
``lazy'': from a behavioural point of view, we're still working with
prenex types.}

Type resolution and importation are tweaked in the obvious way, as follows.
A type is \defn{resolved} if it has no available quantifier, \ie, if
it has the form 
$$T_1\to T_2\to \ldots\to T_n\to X\;\;\;\;\;=\;\;\;\;\;T_1\ldots T_n\to X$$ for
$n\ge 0$ and types $T_i$, called \defn{branches}.  (All we have done is drop the requirement that the $T_i$ be prenex.)
Let $\forall X$ be the leftmost available quantifier in a type $T$,
and let $T^X$ be the result of deleting $\forall X$ from $T$ (\eg\ if
$T\,=\,U\to \forall X.V$ then $T^X\,=\,U\to V$).
Define $$T\;\import\;V\;\;\;=\;\;\; T^X[V/X]\,,$$ the result of
\defn{importing} a type $V$ into $T$, and define \defn{iterated 
importation} $T\,\import\,V_1\ldots V_n$ as before.

The (lazy style) transition system $\Delta_T$ of a system $F$ type
remains essentially unchanged:
\begin{itemize}
\item 
States are system $F$ types, with an additional initial state
$\dummy$\,.
\item
A label is a pair $\langle i,V_1\ldots V_k\rangle$ where $i\ge 1$ is a
\defn{branch choice}, $k\ge 0$ and each $V_i$ is a type, called an \defn{import}.
\item Transitions.  A $1$-labelled transition
$$\dummy\hspace{1ex}\transto{1}\hspace{1ex}T$$ from the initial state
to $T$, and transitions
$$\rule{0ex}{3ex}T_1T_2\ldots T_n\to X\;\;\;\;\;\;\;\transto{\langle
i,V_1\ldots V_k\rangle}\;\;\;\;\;\;\;U_1U_2\ldots U_m\to Y$$ whenever $1\le i\le n$
and
$$T_i\import V_1\ldots V_k\;\;\;=\;\;\;U_1U_2\ldots U_m\to Y$$
\end{itemize}

\section{Black box characterisation of system $F$ terms}

A \defn{black box importation} is an importation of the form
$$\forall X.T\;\import\; X\;\;\;\;\;=\;\;\;\;\; T\,,$$
simply deleting the quantifier.  Thus the bound variable $X$ becomes
free.  We refer to $X$ as a \defn{black box}.  (We continue to assume,
without loss of generality, that within a type all bound variables are
distinct from one another and from all free variables.)
Let $T$ be a closed\footnote{No free variables.} system $F$ type and $d$ a dialogue in the
$\pla$-backtracking game $\pbacktrack{G(\Delta_T)}$.
The first player $\opp$ \defn{imports black boxes} in $d$ if every
importation associated with $\opp$ in $d$ is a black box importation,
and the second player $\pla$ \defn{respects black boxes} in $d$ if
every import associated with $\pla$ takes its free variables among the
black boxes imported hitherto by $\opp$.  A dialogue in which $\opp$
imports black boxes and $\pla$ respects them is a \defn{black box
dialogue}.
The \defn{black box game} $\bbgame{G(\Delta_T)}$ is the
restriction of the $\pla$-backtracking game $\pbacktrack{G(\Delta_T)}$
to black box dialogues.

The copycat condition extends from the lambda calculus case to system
$F$ in the obvious way: the colour of a transition is once again the
rightmost variable of the target.
\begin{theorem}\label{theoremBB}
The $\eta$-expanded $\beta$-normal terms of a closed system $F$ type $T$ are
in bijection with finite live strategies on the black box game
$\bbgame{G(\Delta_T)}$ which satisfy the copycat condition.
\end{theorem}
\begin{proof}
The definability proof in \cite{Hug97}.
\end{proof}

\section{Uniformity by copycat expansion}

The black box game is highly unsymmetric:
\begin{itemize}
\myitem{1} $\pla$ can backtrack, while $\opp$ cannot.
\myitem{2} $\pla$ is subject to the copycat condition, while $\opp$ is not.
\myitem{3} $\opp$ can only import black boxes (free variables); $\pla$ can import 
arbitrary types, so long as their free variables are prior black
boxes.
\end{itemize}
To compose strategies we must symmetrise the game, so that $\opp$ and
$\pla$ can interact.

A symmetrisation of backtracking (1) was obtained by Coquand
\cite{Coq91,Coq95}.  A shared history of two asymmetric strategies is
built, in which both players backtrack.  Each time either asymmetric
strategy plays a move, it can only see a projection of the shared
history in which the opposing player does not backtrack.
This interaction was made lambda-calculus specific by Hyland-Ong
\cite{HO}, who called the projections \emph{views} and called the symmetrised
strategies \emph{innocent}.

Symmetrising the copycat condition (2) will be automatic, coming as a
simple side effect of the views: we simply demand that, in their
respective views, both strategies adhere to the copycat condition.

We shall symmetrise with respect to black boxes (3) via the notion of
\emph{copycat expansion} \cite{Hug06h} recalled below.\footnote{Copycat
expansion was implicit in \cite{Hug00}, occuring during interaction.
In \cite{Hug06h} it was made explicit, being applied to the strategies
prior to ineraction, rather just during interaction.}

Symmetrising (1) yields interaction for lambda calculus over a single
base type symbol.  Symmetrising (1) \& (2) yields interaction for
lambda calculus over a set of base type symbols.  Symmetrising
(1)---(3) yields interaction for system $F$.
\begin{center}
\begin{tabular}{|l|c|}
\hline
\bf Symmetrising & \bf yields interaction for \\ \hline
(1) & $\lambda$, single base type \\ \hline
(1) \& (2) & $\lambda$, set of base types \\ \hline
(1) \& (2) \& (3) & system $F$ \\ \hline
\end{tabular}
\end{center}
We shall refer to the copycat condition together with copycat
expansion as \defn{uniformity}.
A visual summary of the symmetrisation is below, where $T$ is a system $F$ type.
\begin{center}\vspace*{1ex}%
\Rnode{b}{$\bbgame{G(\Delta_T)}$}
\hspace*{12ex}
\Rnode{p}{$\pbacktrack{G(\Delta_T)}$}
\hspace*{12ex}
\Rnode{s}{$\backtrack{G(\Delta_T)}$}%
\ncline[nodesep=.5ex,arrows=->]{b}{p}\naput{uniformity}%
\ncline[nodesep=.5ex,arrows=->]{p}{s}\naput{innocence}%
\end{center}
An arrow here indicates how a strategy on the left lifts to a strategy on the right.

\subsection{Symmetrising black boxes via copycat expansion}

Let $T$ be a closed system $F$ type and $d$ a dialogue in the
$\pla$-backtracking game $\pbacktrack{G(\Delta_T)}$ which satisfies
the copycat condition.
Let $X$ be a black box in $d$, let $U$ be a type, and define $d[U/X]$
as the result of substituting $U$ for free occurrences of $X$ in the
imports of $d$.

\begin{center}
\Large\framebox{To be continued\ldots}
\end{center}

\small
\bibliographystyle{myalphaams}
\bibliography{main}

\end{document}